\documentclass[a4paper,leqno,11pt]{amsart}

\usepackage{epsf,graphicx,epsfig}
\usepackage{amscd}
\usepackage{amsmath,latexsym,amssymb,amsthm}
\usepackage[nospace,noadjust]{cite}
\usepackage{textcomp}
\usepackage{setspace,cite}
\usepackage{lscape,fancyhdr,fancybox}
\usepackage[all,cmtip]{xy}
\usepackage{tikz}
\usetikzlibrary{shapes,arrows,decorations.markings}
\usepackage{graphicx}

\usepackage{color}
\usepackage{url}
\usepackage{enumerate}
\usepackage[mathscr]{euscript}

  \theoremstyle{plain}

\swapnumbers
    \newtheorem{thm}{Theorem}[section]
    \newtheorem{prop}[thm]{Proposition}
   \newtheorem{lemma}[thm]{Lemma}
    \newtheorem{corollary}[thm]{Corollary}
    
    \newtheorem{subsec}[thm]{}
\theoremstyle{definition}
    \newtheorem{defn}[thm]{Definition}
    \newtheorem{exam}[thm]{Example}

\theoremstyle{remark}
     \newtheorem{remark}[thm]{Remark}

\newcommand{\HH}{\mathcal{H}}

\title{}
\author{}
\date{}
\usepackage{amssymb}

\usepackage{hyperref}
\hypersetup{
	colorlinks,
	citecolor=blue,
	filecolor=black,
	linkcolor=blue,
	urlcolor=black
}

\begin{document}
\title{Nambu structures and associated bialgebroids}

\author{Samik Basu}
\email{samik.basu2@gmail.com; mcssb@iacs.res.in}
\address{Department of Mathematical and Computational Science,
Indian Association for the Cultivation of Science, 
Kolkata-700032, India.}

\author{Somnath Basu}
\email{somnath.basu@iiser.ac.in}
\address{Department of Mathematics and Statistics,
Indian Institute of Science Education and Research, Mohanpur-741246,
West Bengal, India.}

\author{Apurba Das}
\email{apurbadas348@gmail.com}
\address{Stat-Math Unit,
Indian Statistical Institute, Kolkata 700108,
West Bengal, India.}

\author{Goutam Mukherjee}
\email{gmukherjee.isi@gmail.com}
\address{Stat-Math Unit,
Indian Statistical Institute, Kolkata 700108,
West Bengal, India.}


\thispagestyle{empty}

\begin{abstract}
This paper investigates higher order generalizations of well known results for Lie algebroids and bialgebroids. It is proved that $n$-Lie algebroid structures correspond to $n$-ary generalization of Gerstenhaber algebras and are implied by $n$-ary generalization of linear Poisson structures on the dual bundle. A Nambu-Poisson manifold (of order $n>2$) gives rise to a special bialgebroid structure which is referred to as a weak Lie-Filippov bialgebroid (of order $n$). It is further demonstrated that such bialgebroids canonically induce a Nambu-Poisson structure on the base manifold. Finally, the tangent space of a Nambu Lie group gives an example of a weak Lie-Filippov bialgebroid over a point. 
\end{abstract}
\maketitle


\vspace{0.5cm}

\section{Introduction}\label{$1$} 

This paper investigates properties of $n$-ary structures on smooth manifolds and smooth vector bundles for $n\geq 3$. Y. Nambu \cite{nambu} introduced a new mechanics which is based on $3$-ary operations, now known as Nambu mechanics. To outline the basic principle of Nambu mechanics, L. Takhtajan \cite{takhtajan}, introduced the notion of a Nambu-Poisson manifold as an $n$-ary generalization of Poisson manifold. Later, Nambu mechanics and properties of Nambu-Poisson manifolds were extensively studied by several authors (\cite{alekseevsky-guha}, \cite{bayen-flato}, \cite{ciccoli}, \cite{gautheron}, \cite{hagi-2002}, \cite{hagi}, \cite{ilmp}, \cite{illmp}, \cite{marmo-vilasi-vinogradov}, \cite{nakanishi}, \cite{wade}, \cite{wang}). A significant difference between a Poisson manifold and a Nambu-Poisson manifold of order greater than 2 is that in the latter case, the associated (Nambu) tensor is locally decomposable. We refer \cite{das-reduction-sing}, \cite{das-reduction-reg}, \cite{das-gon-muk}, \cite{jurco-schupp}, \cite{jurco-schupp-vysoky}, \cite{leon-sardon}, for some recent work on the subject. Filippov \cite{filippov} introduced $n$-Lie algebras as an $n$-ary generalization of Lie algebras. These are also referred to as Filippov algebras. There is growing interest to understand $n$-ary generalizations of the concepts of Lie algebras and Poisson manifolds. J. Grabowski and G. Marmo \cite{gra-mar} introduced an $n$-ary generalization of Lie algebroid, known as $n$-Lie algebroid or Filippov algebroid of order $n$. 
 
It is well known \cite{mac} that Lie algebroid structures on a smooth vector bundle $A$ over $M$ are in one-to-one correspondence with Gerstenhaber brackets on the graded algebra $\Gamma(\Lambda^\bullet A)$. The bracket together with the standard wedge product of multisections makes $\Gamma(\Lambda^\bullet A)$ a Gerstenhaber algebra. In this paper, we prove a version of this result for $n$-Lie algebroids by introducing the concept of $n$-Gerstenhaber algebra as an $n$-ary generalization of a Gerstenhaber algebra. Explicitly, we  prove that there is a one-to-one correspondence between the set of $n$-Lie algebroid structures on $A$ and the set of all $n$-Gerstenhaber brackets on the graded commutative associative algebra $(\Gamma (\Lambda^\bullet A), \wedge)$, making it an $n$-Gerstenhaber algebra (cf. Theorem \ref{classification-filippov-algeb}).

Lie algebroid structures on a smooth vector bundle are also related to linear Poisson structures on its dual. We introduce a notion of linear Nambu structure of order $n$ on a smooth vector bundle and show that if a bundle admits such a structure, then its dual admits the structure of an $n$-Lie algebroid (cf. Theorem \ref{filippov-algebroid-linear-nambu-structure}).

The notion of Lie bialgebroid, introduced by Mackenzie and Xu \cite{mac-xu}, is a generalization of both Poisson manifolds and Lie bialgebras. Recall that a Lie bialgebroid is a pair $(A, A^\ast)$ of Lie algebroids in duality, where the Lie bracket of $A$ satisfies a compatibility condition expressed in terms of the differential $d_\ast$ on $\Gamma (\Lambda^\bullet A)$, i.e., 
$$d_\ast[X, Y] = [d_\ast X, Y] + [X , d_\ast Y],$$ 
for all $X,~ Y \in \Gamma A.$ 

It is well known that if $M$ is a Poisson manifold, then the Lie algebroid structures on $TM$ and $T^\ast M$ form a Lie bialgebroid. On the other hand, if $(A, A^\ast)$ is a Lie bialgebroid over a smooth manifold $M,$ then there is a canonical Poisson structure on the base manifold $M$ \cite{kosmann,mac-xu}.

It is natural to ask the following question:\\ 
{\it Does there exist some notion of bialgebroid associated to a Nambu-Poisson manifold of order $n > 2$?}

It is well known \cite{vaisman} that for a Nambu-Poisson manifold $M$ of order $n \geq 2$, the space $\Omega^1(M)$ of $1$-forms admits an $n$-ary bracket, called Nambu form-bracket, such that the bracket satisfies almost all the properties of an $n$-Lie algebra bracket except that the fundamental identity is satisfied only in a restricted sense if $n > 2.$ We summarize the properties of this bracket to observe that the Nambu form-bracket on  $\Omega^1(M)$, together with the usual Lie algebroid structure on $TM$ gives rise to a notion which we call a {\it weak Lie-Filippov algebroid pair of order $n$}, $n>2$ (cf. Definition~\ref{weak-Lie-Filippov-algebroid-pair}). Examples of such structures also arise from certain Lie algebroids equipped with a Nambu structure (cf. Remark \ref{Nambu-Lie-example}).  

We formulate a notion of {\it Nambu-Gerstenhaber algebra of order $n$} and  prove that weak-Lie-Filippov algebroid pair structures of order $n$, $n>2$, on a smooth vector bundle $A$ over $M$, are in bijective correspondence with Nambu-Gerstenhaber brackets of order $n$ on the graded commutative associative algebra $\Gamma (\Lambda^\bullet A^*)$, where $A^*$ is the dual bundle (cf. Definition \ref{nambu-n-gerstenhaber-algebra}, Theorem \ref{classification-weak-lie-filippov-algebroid-pair}).

We observe that for a Nambu-Poisson manifold $M$ of order $n > 2$, the Nambu-\linebreak Gerstenhaber bracket on $\Omega^\bullet(M)$, extending the Nambu form-bracket on $\Omega^1(M)$ satisfies certain suitable compatibility condition similar to the compatibility condition of a Lie\linebreak bialgebroid (cf. Proposition \ref {compatibility-d-n-gerstenhaber}). This motivates us to introduce a notion of a {\it weak Lie-Filippov bialgebroid structure} of order $n$ on a smooth vector bundle. A weak Lie-Filippov bialgebroid of order $n$ is a weak Lie-Filippov algebroid pair structure of order $n$, satisfying a compatibility condition (cf. Definition \ref{defn-weak-lie-filippov-bialgebroid}). Thus, given a Nambu-Poisson manifold $M$ of order $n > 2$, we conclude that the pair $(TM, T^*M)$ is a weak Lie-Filippov bialgebroid of order $n$ on $TM$ (cf. Corollary \ref{NP-manifold-weak-lie-filippov-bialgebroid}). A weak Lie-Filippov bialgebra of order $n$ is a weak Lie-Filippov bialgebroid of order $n$ over a point.  

In \cite{xu}, Xu proved that there is a one-to-one correspondence between Lie bialgebroid structures on a smooth vector bundle $A$ and strong differential Gerstenhaber algebra structures on $\Gamma (\Lambda^\bullet A).$ We introduce an $n$-ary generalization of a strong differential Gerstenhaber algebra and prove (cf. Proposition \ref{classification-weak-lie-filippov-strong-nambu-gerstenhaber}) a version of the above correspondence for weak Lie-Filippov bialgebroid structures of order $n$ on smooth vector bundles.

It is well known \cite{kosmann,mac-xu} that if $(A, A^\ast)$ is a Lie bialgebroid over a smooth manifold $M,$ then there is a canonical Poisson structure on the base manifold $M$. Therefore, it is natural to expect that if $(A, A^\ast)$ is a weak Lie-Filippov bialgebroid over $M$, then $M$ would admit a canonical Nambu-Poisson structure. We prove that this is indeed the case. Thus, if $(A, A^\ast)$ is a weak Lie-Filippov bialgebroid of order $n$ over a smooth manifold $M,$ then there is a canonical Nambu-Poisson structure of order $n$ on $M$ such that the anchor $a : A \to TM$ of the Lie algebroid $A$ is a morphism of weak Lie-Filippov bialgebroids $(A, A^*) \to (TM, T^*M).$ If there is a morphism $(A,A^\ast) \to (B,B^\ast)$ of weak Lie-Filippov bialgebroids, the corresponding induced Nambu-Poisson structures on $M$ are the same (cf. Theorem \ref{weak-lie-filippov-nambu-poisson-on base}).

In \cite{vaisman}, the author introduced Nambu-Lie groups of order $n$ as an $n$-ary generalization of Poisson Lie groups and determined its infinitesimal form. More explicitly, the author proved that if $G$ is a Nambu-Lie group, then the intrinsic derivative at identity $\delta :\mathfrak g \longrightarrow \Lambda^n \mathfrak g$ is a $1$-cocycle in the Chevalley-Eilenberg complex of the Lie algebra $\mathfrak g$ of $ G$ with respect to the adjoint representation on $\Lambda ^n \mathfrak g$. Moreover, the dual of $\delta$ makes $\mathfrak g^*$ a Filippov algebra of order $n$. Such a pair $(\mathfrak g, \delta)$ is known as Filippov-Lie bialgebra of order $n$ \cite{ciccoli}. In this paper, we will refer to it as {\it Lie-Filippov bialgebra} of order $n$, instead. We show that a Lie-Filippov bialgebra of order $n$ may be described alternatively as a Lie algebra $\mathfrak g$ with its dual $\mathfrak g^\ast$ an $n$-Lie algebra such that the $n$-Lie algebra bracket on $\mathfrak g^\ast$ and the Chevalley-Eilenberg coboundary operator of $\mathfrak g$ satisfies a compatibility condition similar to a Lie bialgebra. Finally, we observe that a Lie-Filippov bialgebra of order $n-$the infinitesimal form of a Nambu-Lie group of order $n$ may be viewed as a weak Lie-Filippov bialgebra of order $n$ and hence, yields example of weak Lie-Filippov bialgebroid of order $n.$  

\begin{remark}\label{terminology-weak}
As one might expect, a proper generalization of Lie-Filippov bialgebra of order $n$ should be a notion of {\it Lie-Filippov bialgebroid} of order $n$. Unfortunately, we do not have any example yet. In this paper, our focus is on Nambu structures on manifolds. Since Nambu structures of order $n~~(n > 2)$ are known to be more rigid than Poisson structures, we end up getting a weaker notion of bialgebroid. This explains the terminology that we use in this paper. 
\end{remark}
Following are some remarks about the relevance of the contents of the present work in mathematical physics.  
\begin{remark}\label{math-phys}
\begin{enumerate}
\item The Lie groupoids and their infinitesimal objects Lie algebroids provide an ideal framework to study internal and external symmetries of both classical and quantum mechanics. Lie algebroids and Lie bialgebroids appear naturally from Poisson manifolds which are known to be the phase spaces for Hamiltonian mechanics. For example, in \cite{etingof-verchenko}, it is shown that Poisson groupoid provides an appropriate general framework in which to study the classical dynamical Yang-Baxter equation. On the other hand, the infinitesimal form of Poisson groupoid, called Lie bialgebroid introduced in \cite{mac-xu} is an important algebraic concept in Mathematical Physics. We refer \cite{landsman} for an outline of the role played by Lie groupoids and Lie algebroids in physics. 

\item Dirac structure is an important ingredient in physics and has a key role to play in classical mechanics (cf. \cite{mukunda-sudarshan}). In \cite{courant}, T. J. Courant studied Dirac structure on manifolds. It is known that given a Lie bialgebroid $(A, A^\ast)$ over a smooth manifold $M,$ the direct sum bundle $A\oplus A^\ast$ carries a Courant algebroid structure, called the Drienfel'd double of the Lie bialgebroid $(A, A^\ast)$ \cite{liu-weinstein-xu}. The same authors also studied Dirac structure on an arbitrary Courant algebroid generalising the notion of Dirac manifold. For example, if we consider the Lie bialgebroid $(TM, T^\ast M)$ associated to a Poisson manifold $(M, \pi)$ then 
$$L : = \mbox{Graph of $\pi^\sharp$} = \{ (\pi^\sharp\alpha, \alpha)|\alpha \in T^\ast M\} \subset TM \oplus T^\ast M$$ defines a Dirac structure on the Courant algebroid $TM \oplus T^\ast M.$ 

\item The phase spaces of Nambu mechanics are Nambu-Poisson manifolds, just like Poisson manifolds are for Hamiltonian mechanics. Therefore, the study of various structures and invariants associated to Nambu-Poisson manifolds would lead to a deeper understanding of Nambu mechanics. 
\end{enumerate}

\end{remark}

{\bf Organization of the paper}: In \S \ref{$2$}, we recall some basic definitions and fix notations. In \S \ref{$3$}, we introduce the notion of $n$-Gerstenhaber algebra and prove Theorem \ref{classification-filippov-algeb}. In \S \ref{$4$}, we introduce linear Nambu structures of order $n$ on a smooth vector bundle $A$ as a generalization of linear Poisson structures and show that if $A$ is equipped with such a structure, then the dual bundle $A^*$ is an $n$-Lie algebroid (Theorem \ref{filippov-algebroid-linear-nambu-structure}). In \S \ref{$5$}, we introduce the notion of weak Lie-Filippov algebroid pair structure of order $n\geq 3$ on a smooth vector bundle. In order to characterize such structures we introduce the notion of Nambu-Gerstenhaber algebra of order $n$ and prove that weak Lie-Filippov algebroid pair structures of order $n\geq 3$ on a smooth vector bundle correspond to Nambu-Gerstenhaber algebras. Moreover, we prove that if $M$ is a Nambu-Poisson manifold of order $n$ ($n\geq 3$), then the pair $(TM, T^*M)$ forms a weak Lie-Filippov algebroid pair of order $n$. In \S \ref{$6$}, we introduce the notion of weak Lie-Filippov bialgebroid of order $n$, strong Nambu-Gerstenhaber algebra of order $n$, characterize weak Lie-Filippov bialgebroid structures of order $n$ on a smooth vector bundle and prove Theorem \ref{weak-lie-filippov-nambu-poisson-on base}. Finally, in \S \ref{$7$}, we establish a connection between weak Lie-Filippov bialgebroids of order $n$ and Nambu-Lie groups of order $n$. The reader may find the following two schematic diagrams useful.\\
\tikzstyle{block}=[draw,rectangle,text width=10em,text centered, minimum height=15mm,node distance=7em]
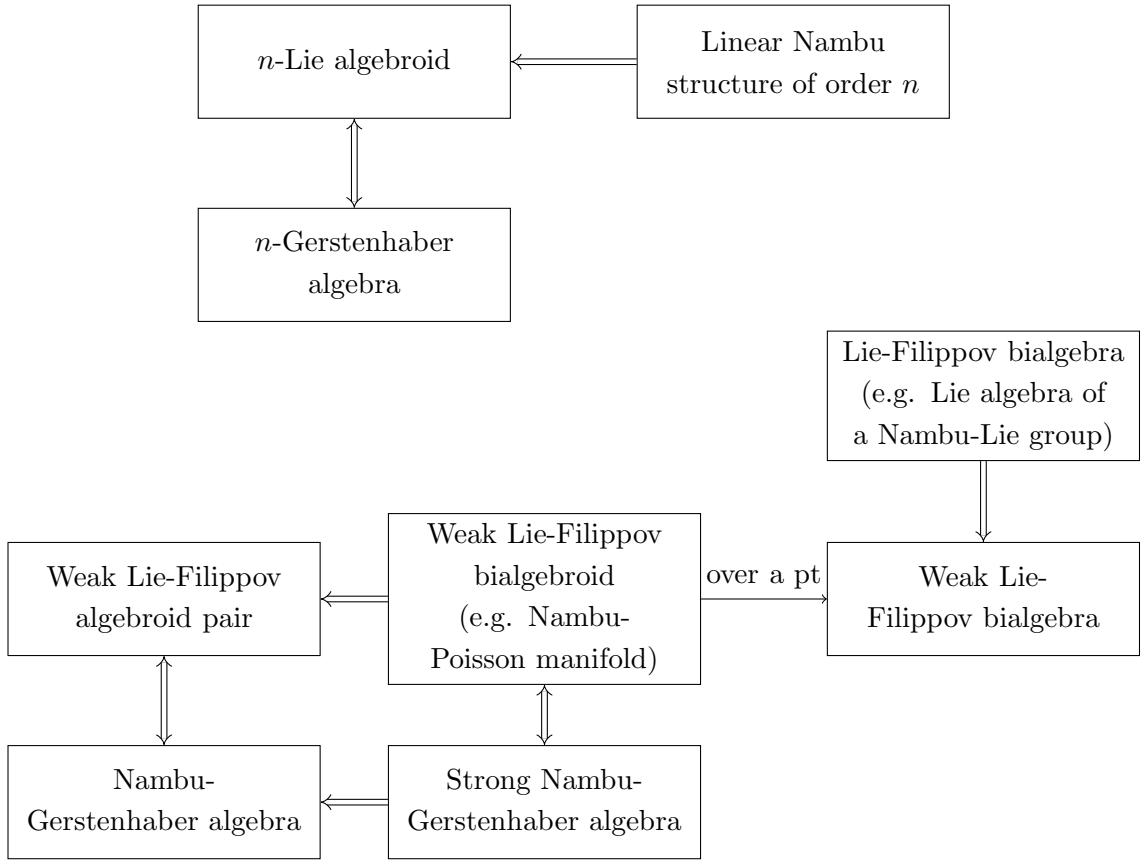
\begin{center}
\begin{tikzpicture}[
implies/.style={double,double equal sign distance,-implies},]
\node[block](nLie){$n$-Lie algebroid};
\node[block, below of =nLie](nGer){$n$-Gerstenhaber algebra};
\node[block, right of =nLie, xshift=8em](linNam){Linear Nambu structure of order $n$};
\draw (linNam) edge[implies] (nLie);
\draw[implies-implies,double equal sign distance](nGer) to (nLie);
\end{tikzpicture}
\end{center}

\begin{center}
\begin{tikzpicture}[
implies/.style={double,double equal sign distance,-implies},]
\node[block](WLFadp){Weak Lie-Filippov algebroid pair };
\node[block, below of =WLFadp](NG){Nambu-Gerstenhaber algebra};
\node[block, right of =NG, xshift=6em](SNG){Strong Nambu-Gerstenhaber algebra};
\node[block, right of =WLFadp, xshift=6em](WLFbd){Weak Lie-Filippov bialgebroid (e.g. Nambu-Poisson manifold)};
\node[block, right of =WLFbd, xshift=8em](WLFb){Weak Lie-Filippov bialgebra};
\node[block, above of =WLFb](LFb){Lie-Filippov bialgebra (e.g. Lie algebra of a Nambu-Lie group)};
\draw[implies-implies,double equal sign distance](WLFadp) to (NG);
\draw[implies-implies,double equal sign distance](WLFbd) to (SNG);
\draw (WLFbd) edge[implies] (WLFadp);
\draw (SNG) edge[implies] (NG);
\draw (LFb) edge[implies] (WLFb);
\draw (WLFbd) edge[->] node[above] {over a pt} (WLFb);
\end{tikzpicture}
\end{center}
%

\vspace{0.5cm}
\section{Preliminaries}\label{$2$}

In this section, we recall some definitions, notations and results. Most of these ideas are from \cite{filippov}, \cite{gra-mar} and \cite{mac}. To simplify notation, we will use either the symbol $\{~, \ldots ,~\}$ or $[~, \ldots, ~]$ to denote an $n$-ary bracket ($n \geq 2$) in various context without causing any confusion. 

Let $\mathfrak{X}(M)$ denote the space of smooth vector fields on a smooth manifold $M$. The notion of a Lie algebroid is a generalization of Lie algebras and tangent bundle of a smooth   manifold. 
\begin{defn}
A {\it Lie algebroid} is a smooth vector bundle $p:A\rightarrow M$ equipped with \\
(i) a Lie algebra structure $[~, ~]$ on the space of $\Gamma A$ of smooth sections of $p$;\\
(ii) a vector bundle map $a : A \rightarrow TM $ over $M$ (called the anchor) such that the induced $C^\infty(M)$-linear map $ a : \Gamma A \rightarrow \mathfrak{X}(M)$ is a map of Lie algebras and satisfies the derivation law $ [X, fY] = f[X, Y] + a (X)(f)Y$ for all $X, Y \in \Gamma A$ and $f \in C^\infty(M).$
\end{defn}

Recall the definition of cohomology of a Lie algebroid with coefficients in a given representation $\nabla$ on a vector bundle $E$. 

\begin{defn}\label{cohomology-lie-algd}
The cohomology of $A$ with coefficients in $\nabla,$ denoted $\HH^\bullet(A,\nabla),$ is the cohomology of the cochain complex $\{\Gamma (\Lambda^\bullet A^\ast\otimes E),d_A\},$ where the differential is given by 
\begin{multline*}
 d_A\phi(X_1,\ldots, X_{k+1})=
\displaystyle{\sum_{r=1}^{k+1}}(-1)^{r+1}\nabla_{X_r}(\phi(X_1,\ldots,X_{r-1},\widehat{X_r},X_{r+1},\ldots,X_{k+1}))\\+
\displaystyle{\sum_{1\leq r<s\leq k+1}}(-1)^{r+s}\phi([X_r,X_s],X_1,\ldots,\widehat{X_r},\ldots,\widehat{X_s},\ldots,X_{k+1}).
\end{multline*}
\end{defn}
In particular when $E= M\times \mathbb R$ and $\nabla = a,$ the anchor of $A$ (called the trivial representation), the above cochain complex is denoted by $\{\Gamma (\Lambda^k A^\ast), d_A\}_{k\geq 0}$ and the corresponding cohomology by $\mathcal{H}^\bullet(A).$

Note that when $A$ is the tangent bundle Lie algebroid, then the coboundary with trivial representation is precisely the de Rham cohomology operator and the corresponding cohomology is the de Rham cohomology, whereas, for $M$ a point space, the cohomology reduces to the Chevalley-Eilenberg cohomology of a Lie algebra.

Recall that a Gerstenhaber (also known as Schouten) algebra is defined as follows. For a graded object  $\mathcal A =\oplus_{i\in \mathbb Z}\mathcal A ^i $, if $a \in  \mathcal A ^i,$ then $|a|$ denotes the degree $i$ of $a$.

\begin{defn}\label{gerstenhaber-algebra}
Let $R$ be a ring and let $C$ be an $R$-algebra. A {\it Gerstenhaber algebra} over $C$ is a $\mathbb Z$-graded commutative associative $C$-algebra $(\mathcal A =\oplus_{i\in \mathbb Z}\mathcal A ^i, \wedge) $ endowed with a $R$-bilinear bracket (called Schouten bracket  or Gerstenhaber bracket)
$$[~ ,~ ] : \mathcal{A}^i \otimes \mathcal{A}^j \to \mathcal{A}^{i+j-1}$$
satisfying
\begin{enumerate}
\item $[a, b] = -(-1)^{(|a|-1)(|b|-1)}[b, a],$
\item $(-1)^{(|a|-1)(|c|-1)}[[a,b],c] + (-1)^{(|b|-1)(|a|-1)}[[b,c],a] + (-1)^{(|c|-1)(|b|-1)}[[c,a],b] =0,$
\item $[a, b\wedge c] = [a, b] \wedge c + (-1)^{(|a|-1)|b|} b \wedge [a, c].$ 
\end{enumerate}
\end{defn}

\begin{remark}\label{condition-2-gerstenhaber-algebra}
Note that in the presence of the condition (1), the condition (2) of the above definition may be re-written as 
$$[a, [b, c]]= [[a, b], c] + (-1)^{(|a|-1)(|b|-1)}[b, [a, c]].$$ 
\end{remark}

For each fixed $a \in \mathcal A^{|a|}$, $[a,~]$ is a derivation with respect to $\wedge$ of degree $(|a|-1)$. For the relabeling of $\mathcal A$ given by $\mathcal A^{(i)} = \mathcal A^{i+1},$ a Gerstenhaber algebra $\mathcal A$ may be viewed as a graded Lie algebra.  

It is a well known fact that given  a Lie algebroid $A$ over a smooth manifold $M$, the algebra $\Gamma(\Lambda^\bullet  A)= \oplus_k\Gamma(\Lambda^k A)$  of multisections endowed with the Schouten bracket  $[~, ~]$ is a Gerstenhaber algebra \cite{mac}. Moreover, the collection of Lie algebroid structures on a smooth vector bundle $A\rightarrow M$ is in one-to-one correspondence with the set of all Gerstenhaber brackets on the graded commutative associative algebra $(\Gamma (\Lambda^\bullet A), \wedge)$ making it a Gerstenhaber algebra.

Recall that the $n$-ary generalization of Lie algebras are called $n$-Lie algebras or Filippov algebras of order $n \geq 2$. 

\begin{defn}\label{n-lie-algebra}
An {\it $n$-Lie algebra} (or a {\it Filippov algebra of order $n$})  is a vector space $\mathcal B$ over $\mathbb R$, together with a $\mathbb R$-multilinear map
$[~,\ldots,~] \colon \underbrace{\mathcal B \times \cdots \times  \mathcal B}_{n\,\textup{times}} \longrightarrow \mathcal B$ 
which is skew-symmetric, i.e., $[a_{\sigma(1)},\ldots,a_{\sigma(n)}] = \textup{sign}(\sigma) [a_1,\ldots,a_n]$ for $\sigma\in\Sigma_n$ and satisfies the generalized Jacobi identity (also called the fundamental identity)
$$[a_1, \ldots, a_{n-1},[b_1, \ldots , b_n]] = \sum_{i=1}^n[ b_1, \ldots, b_{i-1},[a_1,\ldots, a_{n-1},b_i], \ldots,b_n],$$ for all $a_i,~ b_j \in \mathcal B$. 
\end{defn}

Like Lie algebroid we have the notion of an $n$-Lie algebroid defined as follows:
\begin{defn} \label{n-lie-algebroid} 
Let $M$ be a smooth manifold. An {\it{$n$-Lie algebroid}} (or a {\it Filippov algebroid of order $n$}) over $M$, is a smooth vector bundle $p:A \rightarrow M$ such that 
(i) the space of smooth sections $\Gamma A$ admits an $n$-bracket $[~, \ldots ,~]$ which makes $(\Gamma A,  [~, \ldots ,~])$ an $n$-Lie algebra;\\
(ii) there exists a vector bundle morphism $\rho\colon \Lambda^{n-1}A \rightarrow TM$ over $M$ (called the {\it anchor}) satisfying the following conditions:
\begin{displaymath}
[\rho(X_1\wedge \cdots \wedge X_{n-1}), \rho(Y_1\wedge \cdots \wedge Y_{n-1})] = \sum_{i=1}^{n-1}\rho(Y_1\wedge \cdots \wedge [X_1, \ldots , X_{n-1}, Y_i]\wedge \cdots \wedge Y_{n-1});
\end{displaymath}
$$[X_1,\ldots, X_{n-1},fY]= f[X_1,\ldots, X_{n-1}, Y] + \rho(X_1\wedge \cdots \wedge X_{n-1})(f)Y;$$
for all $X_1, \ldots, X_{n-1}, Y_1, \ldots, Y_{n-1}, Y \in \Gamma A$ and $f \in C^{\infty}(M).$
\end{defn}
\noindent Clearly, any $n$-Lie algebra may be considered as an $n$-Lie algebroid over a point with zero anchor. 
\begin{exam}\label{exam-dual-nambu}(\cite{gra-mar}
The tangent bundle $T \mathbb{R}^m \rightarrow \mathbb{R}^m$ has a structure of Filippov algebroid of order $(n+1)$ $(n \leq m)$ uniquely determined by
$$ \big[\frac{\partial}{\partial x_{i_1}}, \ldots, \frac{\partial}{\partial x_{i_{n+1}}} \big] =0 ,$$
and anchor $\bigwedge^{n}T\mathbb{R}^m \rightarrow T \mathbb{R}^m$ is given by the tensor field
$dx_1 \wedge \cdots \wedge dx_n \otimes \frac{\partial}{\partial x_1}$.
\end{exam}

More examples of Filippov algebroids may be found in \cite{gra-mar}.

Recall that a {\it Nambu-Poisson} manifold is a generalization of the notion of {\it Poisson} manifolds and is defined as follows \cite{vaisman}. 
\begin{defn}\label{nambu-poisson manifold}
Let $M$ be a smooth manifold.  A  {\it{Nambu-Poisson bracket}} of order $n$ ($2\leq n \leq ~\mbox{dim}~ M$) on $M$ is an $n$-multilinear
mapping 
$$\{,\ldots,\}\colon C^\infty(M)\times \cdots\times C^\infty(M)
\longrightarrow C^\infty(M)$$ satisfying the following conditions: 
\begin{enumerate}
\item Skew-symmetric: $ \{f_1,\ldots,f_n\}= \textup{sign}(\sigma)\{f_{\sigma(1)},\ldots,f_{\sigma(n)}\}$
for any $\sigma\in \Sigma_n;$
\item Leibniz rule: $\{fg,f_2,\ldots,f_n\}= f\{g ,f_2,\ldots,f_n\}+ g\{f, f_2, \ldots,f_n\};$
\item Fundamental identity: $$\{f_1,\ldots,f_{n-1},\{g_1,\ldots,g_n\}\}=
\displaystyle{\sum_{i=1}^n}\{g_1,\ldots, g_{i-1}, \{f_1,\ldots,f_{n-1},g_i\},\ldots,g_n\}$$
\end{enumerate}
for $f_i,g_j,f, g\in C^\infty(M).$ The pair $(M, \{~, \ldots, ~\})$ is called a {\it{Nambu-Poisson manifold of order $n$}}.  
\end{defn}
\begin{remark}
Note that Poisson manifolds are Nambu-Poisson manifolds of order $2$ \cite {book-vaisman}.  
\end{remark}
\noindent Given a Nambu-Poisson bracket on $M$, there exists an $n$-vector field $P\in\Gamma (\Lambda^n TM),$ called the Nambu-Poisson tensor corresponding to the given bracket, and is defined by $P(df_1,\ldots,df_n)=\{f_1,\ldots,f_n\},$ for $f_1,\ldots,f_n\in C^\infty(M)$. Note that $P$ induces a bundle map 
\begin{align}\label{anchor-NP-manifold}
& P^\sharp: \Lambda^{n-1}T^\ast M\rightarrow TM~~\vspace{5cm}\text{given by}\vspace{5cm} ~~\big\langle \beta, P^\sharp (\alpha_1 \wedge \cdots \wedge \alpha_{n-1})\big\rangle = P (\alpha_1, \ldots , \alpha_{n-1}, \beta),
\end{align}
for all $\alpha_1, \ldots , \alpha_{n-1}, \beta \in \Omega^1(M).$

Given any $(n-1)$ functions $ f_1, \ldots , f_{n-1} \in C^\infty(M),$ the Hamiltonian vector field $X_{f_1 \ldots f_{n-1}}$ associated to these functions is defined by 
$X_{f_1 \ldots f_{n-1}} = P^\sharp (df_1 \wedge \cdots \wedge df_{n-1}).$ In terms of Hamiltonian vector fields, the fundamental identity can be expressed as 

\begin{align}\label{expression-fundamental-id-hamiltonian}  
& [X_{f_1 \ldots f_{n-1}}, X_{g_1 \cdots g_{n-1}}] = \sum_{i=1}^{n-1} X_{g_1 \cdots\{f_1, \ldots , f_{n-1}, g_i\}\cdots g_{n-1}},
\end{align}
for $f_1, \ldots , f_{n-1}, g_1, \ldots , g_{n-1} \in C^\infty(M),$ where the bracket on the left hand side is the Lie bracket of vector fields.
\vspace{0.5cm}
\section{Filippov algebroid and higher order Gerstenhaber algebra}\label{$3$} 

Given  a Lie algebroid $A$ over a smooth manifold $M,$ the algebra $\Gamma(\Lambda^\bullet  A)= \oplus_k\Gamma(\Lambda^k A)$  of multisections endowed with the generalized Schouten bracket  $[~, ~]$ is a Gerstenhaber algebra \cite{mac}. Moreover, the collection of Lie algebroid structures on a smooth vector bundle $A\rightarrow M$ is in one-to-one correspondence with the set of all Gerstenhaber brackets on the graded commutative associative algebra $(\Gamma (\Lambda^\bullet A), \wedge)$ making it a Gerstenhaber algebra. The aim of this section is to prove a version of this result for  Filippov algebroids.

We need a graded version of Filippov algebra of order $n$. 

\begin{defn}\label{graded n-lie algebra}
A {\it graded Filippov algebra} of order $n$  is a $\mathbb Z$-graded real vector space $\mathcal B = \oplus_{i\in \mathbb Z}\mathcal B^i$ equipped with a multilinear bracket
$ [~, \ldots, ~ ]: \mathcal B\times \cdots\times \mathcal B \longrightarrow \mathcal B $ satisfying the following conditions:
\begin{enumerate}
\item $[\mathcal B^{i_1}, \ldots, \mathcal B^{i_n}] \subseteq \mathcal B^{i_1+ \cdots +i_n};$
\item Graded anti-commutativity:
$$[a_1, \ldots , a_i, a_{i+1}, \ldots , a_n]  = -(-1)^{|a_i||a_{i+1}|} [a_1, \ldots , a_{i+1}, a_i, \ldots , a_n],~ a_i \in \mathcal B^{|a_i|};$$
\item Graded fundamental identity:
$$[a_1, \ldots, a_{n-1},[b_1, \ldots , b_n]] = \sum_{i=1}^n[ b_1, \ldots, b_{i-1},[a_1,\ldots, a_{n-1},b_i], \ldots,b_n],$$ where $a_1, a_2, \ldots , a_{n-1} \in \mathcal B^0$ and $b_k \in \mathcal B ^{i_k}$ are homogeneous elements.
\end{enumerate}
\end{defn}
Note that if $\mathcal B$ is a graded Filippov algebra of order $n$, then the bracket $[~, \ldots, ~ ]$ restricted to  $\mathcal B^0$, makes it a Filippov algebra of order $n$.
Next, we introduce a natural $n$-ary generalization of Gerstenhaber algebras, which we call {\it $n$-Gerstenhaber algebras}.

\begin{defn}\label{n-gerstenhaber algebra}
Let $R$ be a ring and $C$ be an $R$-algebra. An {\it $n$-Gerstenhaber algebra} over $C$ is a $\mathbb Z$-graded commutative associative $C$-algebra $(\mathcal A =\oplus_{i\in \mathbb Z}\mathcal A ^i, \wedge) $ with
an $R$-multilinear operation
$$ [~, \ldots , ~] : \mathcal A \times \cdots \times \mathcal A \longrightarrow \mathcal A, ~~ (a_1, \ldots , a_n) \mapsto [a_1, \ldots , a_n],$$ satisfying
\begin{itemize}
\item[(i)] $[a_1, \ldots , a_n] \in \mathcal A^{|a_1|+ \cdots + |a_n| - (n-1)};$
\item[(ii)] $[a_1, \ldots , a_i, a_{i+1}, \ldots,  a_n]= -(-1)^{(|a_i|-1)(|a_{i+1}|-1)} [a_1, \ldots , a_{i+1}, a_i, \ldots,  a_n],$ for all $a_i \in \mathcal A ^{|a_i|}$;
\item[(iii)]  $[a_1, \ldots,  a_{n-1}, [b_1, \ldots , b_n]] 
= \sum _{i=1}^n [b_1, \ldots ,b_{i-1}, [ a_1, \ldots , a_{n-1}, b_i], \ldots , b_n],$\\
for all $a_1, a_2, \ldots , a_{n-1} \in \mathcal A^1$ and $b_k \in \mathcal A^{|b_k|}$;
\item[(iv)] for all  $a_i \in \mathcal A^{|a_i|}$, $b \in \mathcal A^{|b|}$, and $c \in \mathcal A$
$$[a_1, \ldots , a_{n-1}, b\wedge c] = [a_1, \ldots , a_{n-1}, b] \wedge c+(-1)^{ [(|a_1|-1) + \cdots + (|a_{n-1}|-1)]|b|}b\wedge [a_1, \ldots , a_{n-1}, c].$$

\end{itemize}
\end{defn}

The $n$-linear operation in the above Definition will be called an {\it $n$-Gerstenhaber bracket}. 

\begin{remark}
An $n$-Gerstenhaber algebra consists of a triple $(\mathcal A =\oplus_{i\in \mathbb Z}\mathcal A^i, \wedge, [~, \ldots ,~])$ such that 
\begin{enumerate}
\item $(\mathcal A , \wedge)$ is a graded commutative, associative algebra;
\item $(\mathcal A = \oplus_{i\in \mathbb Z}\mathcal A^{(i)}, [~ ,\ldots , ~])$  with $\mathcal A^{(i)} = \mathcal A^{i+1},$ is a graded Filippov algebra of order $n$;
\item for each fixed $a_i \in \mathcal A^{|a_i|},~~ i = 1, 2, \ldots, n-1,$ $[a_1, \ldots , a_{n-1}, ~]$ is a derivation with respect to $\wedge$ of degree $(|a_1|-1) + \cdots +(|a_{n-1}|-1).$ 
\end{enumerate}
\end{remark}
\begin{exam}\label{lie-n-gerstenhaber}
Let $(\mathcal B, [~, \ldots ,~])$ be a Filippov algebra of order $n$. The commutative, associative graded algebra $\Lambda^{\bullet}\mathcal B$ (with respect to the standard wedge product $\wedge$) may be given an $n$-Gerstenhaber algebra structure extending the given bracket defined as follows:
$$ [a_1, \ldots , a_n] = 0, ~~ \mbox{if some}~~ a_j \in \Lambda^0\mathcal B = \mathbb R,$$ 
$$ [a_1, \ldots , a_n] ~~\mbox{ is the given bracket if}~|a_i|=1~\textup{for all}~i.$$ 
Extend the bracket on $\Lambda^\bullet \mathcal B$ uniquely by using the conditions (ii) and (iv) of Definition \ref{n-gerstenhaber algebra}.  
\end{exam}

Note that for a vector bundle $A$ over $M$, the graded $C^\infty(M)$-module  
$$\Gamma (\Lambda^\bullet A) = \oplus_{k\geq 0}\Gamma(\Lambda^k A)$$ is a graded commutative associative algebra with respect to the wedge product of multisections. We have the following result.

\begin{thm}\label{classification-filippov-algeb}
Let $A$ be a smooth vector bundle over $M$. There is a one-to-one correspondence between the set of $n$-Lie algebroid structures on $A$ and the set of all $n$-Gerstenhaber brackets on the graded commutative associative algebra $(\Gamma (\Lambda^\bullet A), \wedge)$ making it an \linebreak$n$-Gerstenhaber algebra.
\end{thm}

\begin{proof}  Let $A \rightarrow M$ be an $n$-Lie algebroid and $ \rho : \Lambda^{n-1}A\to TM$ be the associated anchor map.  We extend the $n$-Lie algebra bracket $[~, \ldots ,~]$ on $\Gamma A = \Gamma (\Lambda^1 A)$ to an $n$-Gerstenhaber bracket $[~, \ldots ,~]$ on $\Gamma (\Lambda^\bullet A) $ as follows. For $X_1, ~X_2, \ldots , X_{n-1} \in \Gamma A$ and $f \in C^\infty(M) = \Gamma(\Lambda^0A),$
$$ [X_1, X_2, \ldots , X_{n-1}, f] : = \rho(X_1\wedge \cdots \wedge X_{n-1})(f).$$
Using the conditions (ii) and (iv) of Definition \ref{n-gerstenhaber algebra}, extend  the bracket $[~, \ldots ,~]$ to all of $\Gamma (\Lambda^\bullet A)$. Observe that the extended bracket on multisections depends only on the local behaviour of the entries. Therefore, the extension is well-defined in the sense that it is independent of concrete realizations of the entries as wedge products. It remains to show that this extended bracket on $\Gamma (\Lambda^\bullet A)$ satisfies condition (iii) of Definition \ref{n-gerstenhaber algebra}.

For $X_1, \ldots X_{n-1}, Y_1, \ldots Y_n \in \Gamma A$ we have the identity
$$[X_1, \ldots, X_{n-1},[Y_1, \ldots , Y_n]] = \sum_{i=1}^n[ Y_1, \ldots, Y_{i-1},[X_1,\ldots, X_{n-1}, Y_i], \ldots,Y_n].$$
Let $f\in C^{\infty}(M)$. From condition (ii) of Definition \ref{n-lie-algebroid} we get
\begin{align*}
&[\rho(X_1\wedge \cdots \wedge X_{n-1}), \rho(Y_1\wedge \cdots \wedge Y_{n-1})](f)\\ 
&= \sum_{i=1}^{n-1}\rho (Y_1 \wedge \cdots \wedge [X_1, \ldots , X_{n-1}, Y_i]\wedge \cdots \wedge Y_{n-1})(f)\\
&= \sum_{i=1}^{n-1}[ Y_1, \ldots, Y_{i-1}, [X_1, \ldots , X_{n-1}, Y_i], \ldots ,Y_{n-1}, f].
\end{align*}

Using the definition of Lie brackets of vector fields we deduce
\begin{align*}
& [X_1, \ldots, X_{n-1}, [Y_1, \ldots, Y_{n-1}, f]]\\
&= \sum_{i=1}^{n-1}[ Y_1, \ldots, Y_{i-1}, [X_1, \ldots , X_{n-1}, Y_i], \ldots ,Y_{n-1}, f] + [Y_1, \ldots, Y_{n-1}, [X_1, \ldots, X_{n-1}, f]].
\end{align*}
Thus, by skew-symmetry, we can say that the property (iii) of Definition \ref{n-gerstenhaber algebra} holds if one of entry $Y_i =f \in C^{\infty}(M) = \Gamma (\Lambda^0A)$.

Let the identity hold for any multisections $Y_j \in \Gamma (\Lambda^{m_j}A)$, $j= 1,2, \ldots , n.$ In other words, for any $X_1, \ldots, X_{n-1} \in \Gamma A$ and $Y_j \in \Gamma (\Lambda^{m_j}A)$, $j= 1,2, \ldots , n$, we assume 
\begin{equation*}
[X_1, \ldots, X_{n-1}, [Y_1, \ldots, Y_{n-1}, Y_n]] = \sum_{i=1}^n[ Y_1, \ldots, Y_{i-1}, [X_1, \ldots , X_{n-1}, Y_i], \ldots , Y_{n-1}, Y_n].
\end{equation*}
We claim that for any $Y\in \Gamma A,$ the quantity $[X_1, \ldots, X_{n-1}, [Y_1, \ldots, Y_{n-1}, Y_n\wedge Y]]$ equals 
$$\sum_{i=1}^{n-1}[ Y_1, \ldots, Y_{i-1}, [X_1, \ldots , X_{n-1}, Y_i], \ldots , Y_{n-1}, Y_n\wedge Y] + [Y_1, \ldots , Y_{n-1}, [X_1, \ldots , X_{n-1}, Y_n \wedge Y]].$$

To prove our claim, we compute either side of the above equality and show that they are equal. Let $Y\in \Gamma A$ and $\theta = ((m_1-1)+ \cdots + (m_{n-1}-1))m_n.$ Applying condition (iv) of Definition \ref{n-gerstenhaber algebra} repeatedly, we get
\begin{align*}
& [X_1, \ldots, X_{n-1}, [Y_1, \ldots, Y_{n-1}, Y_n\wedge Y]]\\
&= [X_1, \ldots , X_{n-1}, [Y_1, \ldots, Y_{n-1}, Y_n]]\wedge Y + [Y_1, \ldots, Y_{n-1}, Y_n]\wedge [X_1, \ldots , X_{n-1}, Y]\\
&+ (-1)^\theta [X_1, \ldots , X_{n-1}, Y_n]\wedge [Y_1, \ldots, Y_{n-1}, Y] +  (-1)^\theta Y_n \wedge [X_1, \ldots , X_{n-1}, [Y_1, \ldots, Y_{n-1}, Y]]. 
\end{align*}
On the other hand,
\begin{align*}
& \sum_{i=1}^{n-1}[ Y_1, \ldots, Y_{i-1}, [X_1, \ldots , X_{n-1}, Y_i], \ldots , Y_{n-1}, Y_n\wedge Y]\\ 
&+ [Y_1, \ldots , Y_{n-1}, [X_1, \ldots , X_{n-1}, Y_n \wedge Y]]\\
&= \sum_{i=1}^{n-1}[ Y_1, \ldots, Y_{i-1}, [X_1, \ldots , X_{n-1}, Y_i], \ldots , Y_{n-1}, Y_n]\wedge Y\\
&+ \sum_{i=1}^{n-1}(-1)^\theta Y_n \wedge [ Y_1, \ldots, Y_{i-1}, [X_1, \ldots , X_{n-1}, Y_i], \ldots , Y_{n-1}, Y]\\
&+ [Y_1, \ldots , Y_{n-1}, [X_1, \ldots , X_{n-1}, Y_n]] \wedge Y +(-1)^\theta [X_1, \ldots , X_{n-1}, Y_n]\wedge [Y_1, \ldots , Y_{n-1}, Y]\\
&+ [Y_1, \ldots , Y_n] \wedge [X_1, \ldots , X_{n-1}, Y] +(-1)^\theta Y_n \wedge [Y_1, \ldots , Y_{n-1}, [X_1, \ldots , X_{n-1}, Y]].
\end{align*}
Comparing terms from the above computations and using the assumption establishes our claim. By induction we may now conclude that the extended bracket satisfies condition (iii) of Definition \ref{n-gerstenhaber algebra}. Hence, $\Gamma (\Lambda^\bullet A) $ has an $n$-Gerstenhaber algebra structure.

Conversely, assume that there is an $n$-Gerstenhaber bracket $[~, \ldots ,~]$ on the graded commutative associative algebra $(\Gamma (\Lambda^\bullet A), \wedge) $ which makes it an $n$-Gerstenhaber algebra. The restriction of this bracket to $\Gamma A$ makes it an $n$-Lie algebra. Let $X_1, ~ \ldots , X_{n-1} \in \Gamma A$ and $f,~g \in C^\infty(M) = \Gamma (\Lambda^0A).$ Observe that by condition (iv) of Definition \ref{n-gerstenhaber algebra}, we have
$$[X_1, \ldots , X_{n-1}, fg] = [X_1, \ldots , X_{n-1}, f]g + f[X_1, \ldots , X_{n-1}, g].$$
Thus, $[X_1, \ldots , X_{n-1}, ~]$ defines a vector field on $M$ given by $ f \mapsto [X_1, \ldots , X_{n-1}, f].$
Hence, we have a linear map 
$ \rho : \Gamma (\Lambda^{n-1} A) \longrightarrow \mathfrak{X}(M),\,\, \rho(X_1 \wedge \cdots \wedge X_{n-1}) = [X_1, \ldots , X_{n-1}, ~].$ 
Note that $\rho(gX_1 \wedge \cdots \wedge X_{n-1})(f) = [X_1, \ldots ,gX_{n-1}, f] = -[X_1, \ldots , X_{n-2}, f, gX_{n-1}]\\
= - [X_1, \ldots , X_{n-2}, f, g]\wedge X_{n-1} - g [X_1, \ldots , X_{n-2}, f, X_{n-1}]
=   g\rho(X_1 \wedge \cdots \wedge X_{n-1})(f).$\\
Therefore, $\rho$ is $C^\infty(M)$-linear. Hence, we may view $\rho :\Lambda^{n-1}A \longrightarrow TM$ as a vector bundle morphism. Let $X_1, \ldots, X_{n-1}, Y_1, \ldots , Y_{n-1} \in \Gamma A$ and $f \in C^\infty (M).$ Then 
\begin{align*}
& [\rho(X_1 \wedge \cdots \wedge X_{n-1}), \rho(Y_1 \wedge \cdots \wedge Y_{n-1})](f)\\
&= [ X_1, \ldots , X_{n-1}, [Y_1, \ldots , Y_{n-1}, f] ] -[ Y_1, \ldots , Y_{n-1}, [X_1, \ldots , X_{n-1}, f] ]\\
&= \sum_{i=1}^{n-1}\rho(Y_1 \wedge \cdots \wedge Y_{i-1} \wedge [X_1, \ldots , X_{n-1}, Y_i]\wedge \cdots \wedge Y_{n-1})(f).
\end{align*}
Thus, 
\begin{align*}
& [\rho(X_1 \wedge \cdots \wedge X_{n-1}), \rho(Y_1 \wedge \cdots \wedge Y_{n-1})]\\
& = \sum_{i=1}^{n-1}\rho(Y_1 \wedge \cdots \wedge Y_{i-1} \wedge [X_1, \ldots , X_{n-1}, Y_i]\wedge \cdots \wedge Y_{n-1}).
\end{align*}

For $Y \in \Gamma A$, 
\begin{align*}
& [X_1, \ldots, X_{n-1}, fY]\\
& = f [X_1, \ldots, X_{n-1}, Y] + [X_1, \ldots, X_{n-1}, f] \wedge Y\\
& = f [X_1, \ldots, X_{n-1}, Y]  + \rho(X_1 \wedge \cdots X_{n-1})(f)Y.
\end{align*}
Therefore, $(A, [~, \ldots, ~], \rho)$ forms an $n$-Lie algebroid.
\end{proof}

\vspace{0.5cm}
\section{Linear Nambu structure}\label{$4$} 
It is well known \cite{mac} that there is a one-to-one correspondence between Lie algebroid structures on a smooth vector bundle $A$ over a smooth manifold $M$ and linear Poisson structures on the dual bundle $A^*.$  In this section, we introduce a notion of  {\it linear Nambu structure} on a smooth vector bundle $A$ and prove that if $A$ is equipped with a linear Nambu structure of order $n,$ then its dual bundle $A^*$ is a Filippov algebroid of order $n$.

Let $p: A \longrightarrow M$ be a smooth vector bundle. Suppose $p^*: A^* \longrightarrow M$ is the dual bundle. Observe that a section $\alpha \in \Gamma A^*$ determines a fibre-wise linear function on $A$, denoted by $l_{\alpha}m,$ and is defined by $$l_{\alpha}(a_m)= \alpha (m)(a_m), ~~a_m \in A_m,~ m\in M.$$ A basic function on $A$ is a function of the type $f \circ p$ where $f \in C^\infty(M).$ For brevity, we refer to fibre-wise linear functions on $A$ as linear functions.

\begin{defn}\label{linear-n-vector-field}
Let $p: A \longrightarrow M $ be a smooth vector bundle. An $n$-vector field $\Pi \in \mathcal{X}^n (A)$ is called {\bf fibrewise linear} or {\bf linear} $n$-vector field on $A$
if the induced $n$-ary bracket on $C^\infty(A)$ satisfies the following properties.
\begin{itemize}
\item[(i)] The bracket of $n$ linear functions is again a linear function.
\item[(ii)] The bracket of $(n-1)$ linear functions and a basic function is a basic function.
\item[(iii)] The bracket of functions is zero if there are more than one basic functions as entries.
\end{itemize}
\end{defn}

\begin{defn}\label{linear-nambu-structure}
Let $p: A \longrightarrow M $ be a smooth vector bundle. A Nambu structure of order $n$ on the total space $A$ is called {\bf fibrewise linear Nambu structure} or {\bf linear Nambu structure} of order $n$ if the Nambu tensor is linear.
\end{defn}

\begin{exam}\label{exam-lin-nam}
 Let $(\mathcal{B}, [~,\ldots,~])$ be a Filippov algebra of order $n.$  Assume that the vector space dimension of $\mathcal{B}$ is $m.$ For a given basis $\{ x_1, \ldots, x_m \}$ of $\mathcal{B}$, regarded as a basis of linear functions on $\mathcal{B}^*$, we can define an $n$-vector field $\Pi_{\mathcal{B}^*}$ on $\mathcal{B}^*$ as (see \cite{gra-mar})
$$ \Pi_{\mathcal{B}^*} = \sum_{i_1,...,i_n} [x_{i_1}, \ldots, x_{i_n}] \frac{\partial}{\partial x_{i_1}} \otimes \cdots \otimes \frac{\partial}{\partial x_{i_n}} .$$
Note that the $n$-vector field $\Pi_{\mathcal{B}^*}$ is linear. In general $\Pi_{\mathcal{B}^*}$ need not be Nambu tensor (that is, need not be locally decomposable). If
dim $\mathcal{B} = m \leqslant n+1$, then $\Pi_{\mathcal{B}^*}$ is decomposable \cite{marmo-vilasi-vinogradov}, hence, defines a linear Nambu structure of order $n$.
\end{exam}

Let $p: A \longrightarrow M$ be a smooth vector bundle with a linear Nambu structure $(A, \{~, \ldots , ~ \})$ of order $n$ on $A$. Let $ \alpha_1, ~\alpha_2, \ldots , ~\alpha_n \in \Gamma A^*.$ Then $l_{\alpha_i},~ i \in \{1, \ldots, n\},$ are linear functions on $A$. 

Let  $[ \alpha_1, \ldots , \alpha_n] \in \Gamma A^*$ be such that
\begin{align}\label{lin-func-brck} 
& \{ l_{\alpha_1}, \ldots, l_{\alpha_n}\} = l_{[ \alpha_1, \ldots , \alpha_n]}.
\end{align}
Observe that for any $\alpha_1, \ldots , \alpha_{n-1} \in \Gamma A^*$ and $f \in C^\infty (M),$ $\{l_{\alpha_1}, \ldots , l_{\alpha _{n-1}}, f \circ p\}$ is a basic function. Moreover, it is straight forward to check that $\{l_{\alpha_1}, \ldots , l_{\alpha _{n-1}}, f \circ p\}$ is a derivation with respect to $f$.  Therefore, there exists a vector field $ \rho (\alpha_1 \wedge \cdots \wedge \alpha_{n-1})$ such that
$\{l_{\alpha_1}, \ldots , l_{\alpha _{n-1}}, f \circ p\} = \rho (\alpha_1 \wedge \cdots \wedge \alpha_{n-1})(f) \circ p.$

Note that for $g \in C^\infty(M)$,
$$\{l_{\alpha_1}, \ldots , l_{g\alpha _{n-1}}, f \circ p\} = \rho (\alpha_1 \wedge \cdots \wedge  g\alpha_{n-1})(f) \circ p.$$ On the other hand,
\begin{align*}
\{l_{\alpha_1}, \ldots , l_{g\alpha _{n-1}}, f \circ p\}
&= \{l_{\alpha_1}, \ldots , (g \circ p) l_{\alpha _{n-1}}, f \circ p\}~~(\mbox{because,}~~l_{g\alpha} = (g\circ p)l_{\alpha})\\
&= (g \circ p) \{l_{\alpha_1}, \ldots , l_{\alpha _{n-1}}, f \circ p\}+ l_{\alpha_{n-1}}\{l_{\alpha_1}, \ldots, g\circ p, f\circ p\}\\
&= (g \circ p) \rho (\alpha_1 \wedge \cdots \wedge  \alpha_{n-1})(f) \circ p\\
&= (g \rho (\alpha_1 \wedge \cdots \wedge  \alpha_{n-1})(f)) \circ p.
\end{align*}
Thus, $\rho$ is  $C^\infty(M)$-linear,  hence, yields a vector bundle map 
$\rho : \Lambda^{n-1}A^* \longrightarrow TM.$

\begin{thm}\label{filippov-algebroid-linear-nambu-structure}
Let $p: A \longrightarrow M$ be a smooth vector bundle with a linear Nambu structure $(A, \{~, \ldots , ~ \})$ of order $n$ on $A$. For $ \alpha_1, ~\alpha_2, \ldots , ~\alpha_n \in \Gamma A^*,$  let  $[ \alpha_1, \ldots , \alpha_n] \in \Gamma A^*$ be such that 
$\{ l_{\alpha_1}, \ldots, l_{\alpha_n}\} = l_{[ \alpha_1, \ldots , \alpha_n]}.$ Then this bracket on $\Gamma A^*$ together with the vector bundle map $\rho : \Lambda^{n-1}A^* \longrightarrow TM$ as described above defines a Filippov algebroid structure of order $n$ on the dual bundle $A^*$.
\end{thm}
\begin{proof}
The bracket $[~,\ldots, ~]$ on $\Gamma A^*$ is skew symmetric and satisfies fundamental identity as so does the bracket $\{~, \ldots, ~\}.$
Let $\alpha_1, \ldots , \alpha_{n-1}, \beta_1, \ldots, \beta_{n-1} \in \Gamma A^*$ and $f \in C^\infty(M).$
Note that
\begin{align*}
& \sum_{i=1}^{n-1}\{l_{\beta_1}, \ldots, l_{[\alpha_1, \ldots,\alpha_{n-1}, \beta_i]}, \ldots, l_{\beta_{n-1}}, f \circ p\}\\
&= \sum_{i=1}^{n-1}\rho(\beta_1 \wedge \cdots \wedge \beta_{i-1}\wedge [\alpha_1, \ldots, \alpha_{n-1}, \beta_i]\wedge \cdots\wedge\beta_{n-1})(f)\circ p.
\end{align*}
On the other hand,
\begin{align*}
& \sum_{i=1}^{n-1}\{l_{\beta_1}, \ldots, l_{[\alpha_1, \ldots,\alpha_{n-1}, \beta_i]}, \ldots, l_{\beta_{n-1}}, f \circ p\}\\
&= \sum_{i=1}^{n-1}\{l_{\beta_1}, \ldots, \{l_{\alpha_1}, \ldots, l_{\alpha_{n-1}}, l_{\beta_i}\}, \ldots, l_{\beta_{n-1}}, f \circ p\}\\
&= \{l_{\alpha_1}, \ldots , l_{\alpha_{n-1}}, \{l_{\beta_1}, \ldots , l_{\beta_{n-1}}, f \circ p\}\}-\{l_{\beta_1}, \ldots ,  l_{\beta_{n-1}}, \{l_{\alpha_1}, \ldots , l_{\alpha_{n-1}}, f \circ p\}\}\\
&= \rho(\alpha_1 \wedge \cdots \wedge\alpha_{n-1})\rho(\beta_1 \wedge \cdots \wedge \beta_{n-1})(f)\circ p- \rho(\beta_1\wedge \cdots \wedge \beta_{n-1}) \rho(\alpha_1 \wedge \cdots \wedge \alpha_{n-1})(f)\circ p\\
&= [\rho(\alpha_1 \wedge \cdots \wedge \alpha_{n-1}), \rho(\beta_1 \wedge \cdots \wedge \beta_{n-1})](f)\circ p.
\end{align*}
Therefore,
\begin{align*}
& [\rho(\alpha_1 \wedge \cdots \wedge \alpha_{n-1}), \rho(\beta_1 \wedge \cdots \wedge \beta_{n-1})]\\
&= \sum_{i=1}^{n-1}\rho(\beta_1 \wedge \cdots \wedge \beta_{i-1}\wedge  [\alpha_1, \ldots, \alpha_{n-1}, \beta_i]\wedge \cdots\wedge\beta_{n-1}).
\end{align*}

It remains to show that
$$[\alpha_1, \ldots, \alpha_{n-1}, f\beta] = f[\alpha_1, \ldots , \alpha_{n-1}, \beta] + \rho(\alpha_1 \wedge \cdots \wedge \alpha_{n-1})(f)\beta.$$
This follows from the identity
\begin{align*}
l_{[\alpha_1, \ldots, \alpha_{n-1}, f\beta]} &= \{l_{\alpha_1}, \ldots , l_{\alpha_{n-1}}, l_{f\beta}\}\\
&= \{l_{\alpha_1}, \ldots , l_{\alpha_{n-1}}, (f\circ p)l_{\beta}\}\\
&= (f \circ p)\{l_{\alpha_1}, \ldots , l_{\alpha_{n-1}}, l_{\beta}\}+ \{l_{\alpha_1}, \ldots , l_{\alpha_{n-1}}, f\circ p\}l_{\beta}\\
&= (f\circ p) l_{[\alpha_1, \ldots, \alpha_{n-1}, \beta]}+(\rho(\alpha_1 \wedge \cdots \wedge \alpha_{n-1})(f)\circ p)l_{\beta}\\
&= l_{f[\alpha_1, \ldots, \alpha_{n-1}, \beta] }+l_{\rho(\alpha_1 \wedge \cdots \wedge \alpha_{n-1})(f)\beta}\\
&= l_{ f[\alpha_1, \ldots, \alpha_{n-1}, \beta] +\rho(\alpha_1 \wedge \cdots \wedge \alpha_{n-1})(f)\beta}.
\end{align*}
Therefore, $(A^*, [~, \ldots, ~], \rho)$  is a Filippov algebroid of oder $n$. 
\end{proof}

\begin{remark}
It may be noted that linear Nambu structure of order $n$ is a generalization of the notion of linear Poisson structure and for $n > 2$, the converse of the above result need not be true. In fact, given a Filippov algebroid structure on $A^* \rightarrow M$, the total space of the bundle $A \rightarrow M$ carries an $n$-vector field $\Pi_{A}$ such that the induced bracket
on the linear functions on $A$ is given by Equation (\ref{lin-func-brck}). Thus, the space of linear functions on $A$ forms a Filippov algebra of order $n$.
Moreover, one can observe that $\Pi_{A}$ is indeed a linear $n$-vector field in the sense of Definition \ref{linear-n-vector-field}, although $\Pi_A$ need not be a Nambu tensor.
In \cite{marmo-vilasi-vinogradov}, the authors showed that if $\mathcal{B}$ is a Filippov algebra of order $n$ (considered as a Filippov algebroid over a point), then the linear
$n$-vector field on $\mathcal{B}^*$ is a Nambu tensor if and only if the vector space dimension of $\mathcal{B}$ is less or equal to $n+1$ (see Example \ref{exam-lin-nam}). The linear
$n$-vector field on the dual of the Filippov algebroid is given by
$$\Pi_{T^* \mathbb{R}^m} = \frac{\partial}{\partial \dot{x}_1} \wedge \cdots \wedge \frac{\partial}{\partial \dot{x}_n} \wedge \frac{\partial}{\partial x_1} ,$$
where $\dot{x}_1, \ldots, \dot{x}_m$ are the cotangent coordinates of $T^* \mathbb{R}^m$ (cf. Example \ref{exam-dual-nambu}). This defines a linear Nambu structure on $T^* \mathbb{R}^m$.
\end{remark}

\vspace{0.5cm}
\section{Weak Lie-Filippov algebroid pair}\label{$5$}
The aim of this section is to introduce a notion of {\it weak Lie-Filippov algebroid pair} structure on a smooth vector bundle and to show that such structures arise naturally from Nambu-Poisson manifolds of order $n\geq 3.$ The content of this section is a prelude to the notion of weak Lie-Filippov bialgebroid which will develop in the next section.

Let (M, \{~, \ldots, ~\}) be a Nambu-Poisson manifold of order $n\geq 3$. Let $P\in \Gamma (\Lambda^nTM)$ be the corresponding Nambu-Poisson tensor defined by
\[P(df_1,\ldots,df_n)=\{f_1,\ldots,f_n\},~~\text{for all}~~ f_1,\ldots,f_n\in C^\infty(M).\]

Then we have an $n$-ary operation $[~, \ldots, ~]$ (\cite{vaisman}), called {\it Nambu form-bracket}  on the space of $1$-forms defined as follows.
\begin{defn}\label{Nambu-form-bracket}
The Nambu form-bracket is defined as 
\begin{align*}
[\alpha_1, \ldots, \alpha_n]
&= d(P(\alpha_1, \ldots, \alpha_n)) + \sum_{k=1}^n (-1)^{n+k}\iota_{P^\sharp(\alpha_1\wedge \cdots \wedge\widehat{\alpha}_k\wedge \cdots \wedge \alpha_n)}d\alpha_k\\
&=  \sum_{k=1}^n (-1)^{n+k} \mathcal L_{P^\sharp(\alpha_1\wedge \cdots \wedge \widehat{\alpha}_k\wedge \ldots \wedge \alpha_n)}\alpha_k-(n-1)d(P(\alpha_1, \ldots ,\alpha_n)),
\end{align*}
where $\alpha_k,~k=1, \ldots, n$ are $1$-forms on $M$ and $\mathcal L_X$ denotes the Lie derivative operator with respect to a vector field $X$.
\end{defn}

The following is an illustrative example.
\begin{exam}\label{example-nambu-form-bracket}
Consider $M = \mathbb{R}^3$ with the Nambu $3$-tensor $P = x  \frac{\partial}{\partial x}  \wedge \frac{\partial}{ \partial y} \wedge \frac{\partial}{ \partial z}.$ Using the definition
above, one can write down the expressions of Nambu form-brackets of the basis elements of the space of $1$-forms. For instance,
\begin{align*}
 &[f_1 dx, f_2 dx, g dy]\\
=&~ f_2 g ~\iota_{P^\sharp(dx \wedge dy)} d (f_1 dx) - f_1 g ~\iota_{P^\sharp (dx \wedge dy)} d (f_2 dx) \\
=&~ x f_2 g ~ \iota_{ \frac{\partial}{\partial z}} (df_1 \wedge dx) - x f_1 g ~\iota_{\frac{\partial}{\partial z}} (df_2 \wedge dx) 
= x g ~ ( f_2  \frac{\partial f_1}{\partial z}  - f_1 \frac{\partial f_2}{\partial z}) ~dx
\end{align*}
and
\begin{align*}
 &[f dx, g dy, hdz]\\
=&~ d (fgh P (dx, dy, dz)) + gh ~\iota_{P^\sharp (dy \wedge dz)} d(fdx) - fh ~ \iota_{P^\sharp (dx \wedge dz)} d (g dy) + fg ~ \iota_{P^\sharp (dx \wedge dy)} d (h dz)\\
=&~ d (fgh x) + xgh ~ \iota_{ \frac{\partial}{\partial x} } (df \wedge dx) + xfh ~ \iota_{\frac{\partial}{ \partial y}} (dg \wedge dy) 
+ xfg ~ \iota_{\frac{\partial}{ \partial z}} (dh \wedge dz) \\
=&~ d (fgh x) + xgh ~( - \frac{\partial f }{\partial y} dy - \frac{\partial f}{\partial z} dz   ) + xfh ~ (- \frac{\partial g }{\partial x} dx - \frac{\partial g}{\partial z} dz ) + xfg ~(- \frac{\partial h }{\partial x} dx - \frac{\partial h}{\partial y} dy )\\
=&~ fgh ~dx  + x \frac{\partial f}{\partial x} gh ~dx + x \frac{\partial g}{\partial y} fh ~dy + x \frac{\partial h}{\partial z} fg ~dz,
\end{align*}
where $f$, $g$, $h$, $f_1$ and $f_2$ are smooth functions.
Explicit expression of all other brackets can be obtained similarly.
\end{exam}

We have the following proposition from \cite{vaisman}.

\begin{prop}\label{properties-Nambu-form-bracket}
For a Nambu-Poisson manifold $(M, \{~, \ldots, ~\})$ of order $n$, the Nambu form-bracket on the space of $1$-forms satisfies the following properties.
\begin{enumerate}
\item The Nambu form-bracket is skew-symmetric;
\item $[df_1, \ldots, df_n] = d\{f_1, \ldots, f_n\}$ for all $f_1, \ldots, f_n \in C^\infty(M)$;
\item for any $1$-forms $\alpha_1, \ldots, \alpha_n$ and for all $f \in C^\infty(M)$,
$$[\alpha_1, \ldots, \alpha_{n-1}, f\alpha_n] =f [\alpha_1, \ldots, \alpha_{n-1}, \alpha_n] + P^\sharp(\alpha_1, \ldots, \alpha_{n-1})(f)\alpha_n;$$
\item for any $1$-form $\alpha$ and for all $f_1, \ldots, f_{n-1} \in C^\infty(M),$
$$[df_1, \ldots, df_{n-1}, \alpha] = \mathcal L_{X_{f_1 \ldots f_{n-1}}}\alpha;$$
\item for any $1$-forms $\alpha_1, \ldots, \alpha_n$ and for all $f_1, \ldots, f_{n-1} \in C^\infty(M)$,
$$\mathcal L_{X_{f_1 \ldots f_{n-1}}}[\alpha_1, \ldots, \alpha_n]= \sum_{i=1}^n[\alpha_1, \ldots, \mathcal L_{X_{f_1 \ldots f_{n-1}}}\alpha_i,\ldots, \alpha_n].$$
\end{enumerate}
\end{prop}
From the properties $(4)$ and $(5)$ of Proposition \ref{properties-Nambu-form-bracket} we get
\begin{corollary}\label{fundamental identity-for exact form}
For any $1$-forms $\alpha_1, \ldots, \alpha_n$ and for all $f_1, \ldots, f_{n-1} \in C^\infty(M)$, the following identity holds
$$[df_1, \ldots, df_{n-1}, [\alpha_1, \ldots, \alpha_n]]= \sum_{i=1}^n [ \alpha_1, \ldots, \alpha_{i-1}, [df_1, \ldots, df_{n-1}, \alpha_i], \ldots, \alpha_n].$$
\end{corollary}
\begin{remark}\label{fundamental-identity-NP-manifold}
Observe that any closed form is locally exact and by property $(3)$ of Proposition \ref{properties-Nambu-form-bracket} the $n$-ary bracket depends only on the local behaviour of its entries. Thus, it follows from Corollary \ref{fundamental identity-for exact form} that the fundamental identity
$$[\alpha_1, \ldots , \alpha_{n-1}, [\beta_1, \ldots , \beta_n]] = \sum_{i=1}^n [\beta_1, \ldots , \beta_{i-1}, [\alpha_1, \ldots , \alpha_{n-1}, \beta_i], \ldots ,\beta_n]$$ holds whenever the $1$-forms $\alpha_j,~ 1\leq j \leq n-1$ are closed. 
\end{remark}

In order to understand the above properties of Nambu-Poisson manifolds in the proper perspective, we introduce a notion of a {\it weak-Lie-Filippov algebroid pair} structure on a smooth vector bundle. 

\begin{defn}\label{weak-Lie-Filippov-algebroid-pair}
A smooth vector bundle $A\longrightarrow M$ is said to admit a weak Lie-Filippov algebroid pair structure of order $n>2$ if the following conditions are satisfied.
\begin{enumerate}
\item The vector bundle  $A\longrightarrow M$ is a Lie algebroid.  
\item The space of smooth sections $\Gamma A^*$ of the dual bundle $A^*$ admits a skew-symmetric $n$-ary bracket
$$[~, \ldots ,~]: \underbrace{\Gamma A^* \times \cdots \times \Gamma A^*}_{n~ \mbox{times}} \longrightarrow \Gamma A^*$$
satisfying 
$$[\alpha_1, \ldots , \alpha_{n-1}, [\beta_1, \ldots , \beta_n]] = \sum_{i=1}^n [\beta_1, \ldots , \beta_{i-1}, [\alpha_1, \ldots , \alpha_{n-1}, \beta_i], \ldots ,\beta_n]$$
for all $d_A$-closed sections $\alpha_i \in \Gamma A^*,~ 1\leq i \leq n-1$ and for any sections $\beta_j \in \Gamma A^*,~ 1\leq j\leq n,$ $d_A$ being the differential of the cochain complex $\{\Gamma(\Lambda^kA^\ast), d_A \}_{k\geq 0}$ (cf. Definition \ref{cohomology-lie-algd}) defining the Lie algebroid cohomology of $A$ with trivial representation.
\item There exists a vector bundle morphism $\rho : \Lambda^{n-1}A^* \longrightarrow TM$, called the anchor of the pair $(A, A^*)$, such that the identity
$$ [\rho (\alpha_1 \wedge \cdots \wedge \alpha_{n-1}), \rho (\beta_1 \wedge \cdots \wedge \beta_{n-1})] = \sum_{i=1}^{n-1}\rho (\beta_1 \wedge \cdots \wedge [\alpha_1, \ldots , \alpha_{n-1}, \beta_i] \wedge \cdots \wedge \beta_{n-1})$$
holds for all $d_A$-closed sections $\alpha_i \in \Gamma A^*,~ 1\leq i \leq n-1$ and for any sections $\beta_j \in \Gamma A^*,~ 1\leq j\leq n-1.$
\item For all sections $\alpha_i \in \Gamma A^*,~ 1\leq i \leq n$ and any $f \in C^\infty(M)$,
$$[\alpha_1, \ldots , \alpha_{n-1}, f\alpha_n] = f [\alpha_1, \ldots , \alpha_{n-1}, \alpha_n] + \rho (\alpha_1 \wedge \cdots \wedge \alpha_{n-1})(f)\alpha_n$$ holds.
\end{enumerate}
A weak Lie-Filippov algebroid pair structure of order $n>2$ on a vector bundle $A$ over $M$ will be denoted by the pair $(A, A^*).$
\end{defn}

\begin{remark}
For a weak Lie-Filippov algebroid pair $(A,A^*)$ over $M$, $A$ is a Lie algebroid and $A^*$ is almost a Filippov algebroid over $M$ except that some of the defining conditions are restricted (cf. Definition \ref{n-lie-algebroid}).
\end{remark}
Next, we prove a property of weak Lie-Filippov algebroid pair structure similar to that of Filippov algebroids (cf. Theorem \ref{classification-filippov-algeb}). For this, we first introduce the following definition.

\begin{defn}\label{nambu-n-gerstenhaber-algebra}
Let $R$ be a ring and $C$ be an $R$-algebra. A {\it Nambu-Gerstenhaber algebra of order $n$} over $C$ is a $\mathbb Z$-graded commutative associative $C$-algebra $(\mathcal A =\oplus_{i\in \mathbb Z}\mathcal A ^i, \wedge) $ equipped with 
\begin{enumerate}
\item a differential $\delta$ on $\mathcal A$, which is a derivation of degree $1$ with respect to $\wedge$, that is, $\delta : \mathcal A^i \longrightarrow \mathcal A^{i+1}$ is a $R$-linear map such that $ \delta^2 =0$ and 
$$ \delta (a \wedge b) = \delta a \wedge b + (-1)^{|a|}a\wedge \delta b$$ for homogeneous elements $ a, b \in \mathcal A;$
\item an $R$-multilinear operation
$$ [~, \ldots , ~] : \mathcal A \times \cdots \times \mathcal A \longrightarrow \mathcal A, ~~ (a_1, \ldots , a_n) \mapsto [a_1, \ldots , a_n],$$ referred to as a Nambu-Gerstenhaber $n$-ary bracket, satisfying
\begin{itemize}
\item[(i)] $[a_1, \ldots , a_n] \in \mathcal A^{|a_1|+ \cdots + |a_n| - (n-1)},$ for homogeneous entries;
\item[(ii)] $[a_1, \ldots , a_i, a_{i+1}, \ldots,  a_n]= -(-1)^{(|a_i|-1)(|a_{i+1}|-1)} [a_1, \ldots , a_{i+1}, a_i, \ldots,  a_n],$ for all $a_i \in \mathcal A ^{|a_i|}$;
\item[(iii)]  $[a_1, \ldots,  a_{n-1}, [b_1, \ldots , b_n]] 
= \sum _{i=1}^n [b_1, \ldots ,b_{i-1}, [ a_1, \ldots a_{n-1}, b_i], \ldots b_n],$\\
for all $\delta$-closed $a_1, a_2, \ldots , a_{n-1} \in \mathcal A^1$ and $b_k \in \mathcal A^{|b_k|}$;
\item[(iv)] $[a_1, \ldots , a_{n-1}, b\wedge c]$ 
$$= [a_1, \ldots , a_{n-1}, b] \wedge c +(-1)^{ [(|a_1|-1) + \cdots + (|a_{n-1}|-1)]|b|}b\wedge [a_1, \ldots , a_{n-1}, c]$$
holds for all  $a_i \in \mathcal A^{|a_i|}$, $b \in \mathcal A^{|b|}$, and $c \in \mathcal A.$
\end{itemize}
\end{enumerate}
\end{defn}
A Nambu-Gerstenhaber algebra of order $n$ will be denoted by the quadruple 
$$(\mathcal A, \wedge, [~, \ldots , ~], \delta).$$ 

\begin{thm}\label{classification-weak-lie-filippov-algebroid-pair}
There is a one-to-one correspondence between weak Lie-Filippov algebroid pair structures of order $n>2$ on a smooth vector bundle $A\longrightarrow M$ and Nambu-Gerstenhaber algebra structures of order $n$ on the graded $C^\infty(M)$-module $\Gamma (\Lambda^\bullet A^*).$  
\end{thm}

\begin{proof}
Let $A\longrightarrow M$ be a smooth vector bundle. Assume that $(A, A^*)$ is a weak Lie-Filippov algebroid pair of order $n>2$. We extend the given $n$-ary bracket $[~, \ldots ,~]$ on $\Gamma A^*$ to an $n$-ary bracket on $\Gamma (\Lambda^\bullet A^*)$ as follows. For $\alpha_1, \ldots , \alpha_{n-1} \in \Gamma A^*$ and $f \in C^\infty(M) = \Gamma (\Lambda^0A^*),$ we define
$$[\alpha_1, \ldots ,\alpha_{n-1}, f] := \rho ( \alpha_1 \wedge \cdots \wedge \alpha_{n-1})(f),$$ 
where $\rho$ is the anchor of the pair $(A, A^*).$ We extend the bracket $[~, \ldots ,~]$ uniquely to $\Gamma (\Lambda^\bullet A^*)$ by induction using conditions $2$ (ii) and $2$ (iv) of Definition \ref{nambu-n-gerstenhaber-algebra}. Then by an argument similar to the first half of the proof of Theorem \ref{classification-filippov-algeb}, one shows that condition $2$ (iii) of Definition \ref{nambu-n-gerstenhaber-algebra} holds. Thus, $(\Gamma (\Lambda^\bullet A^*), \wedge , [~, \ldots , ~], d_A)$ is a Nambu-Gerstenhaber algebra of order $n$.

Conversely, suppose that $(\Gamma (\Lambda^\bullet A^*), \wedge , [~, \ldots , ~], \delta)$ is a Nambu-Gerstenhaber algebra of order $n~(n > 2)$. Then the differential $\delta$ induces a Lie algebroid structure on $A$ with $d_A = \delta$ as the differential operator for the Lie algebroid cochain complex \cite{xu}. The restriction of the bracket $[~, \ldots ,~]$ to $\Gamma A^*$ satisfies condition $(2)$ of Definition \ref{weak-Lie-Filippov-algebroid-pair}. Next, observe that for $ \alpha_1, \ldots ,\alpha_{n-1} \in \Gamma A^*$ and $f, g \in C^\infty(M),$ we have 
$$[\alpha_1, \ldots ,\alpha_{n-1}, fg] = [\alpha_1, \ldots ,\alpha_{n-1}, f]g + f[\alpha_1, \ldots ,\alpha_{n-1}, g].$$
Thus, $[\alpha_1, \ldots ,\alpha_{n-1}, ~]$ defines a vector field on $M$, giving us a $C^\infty(M)$-linear map 
$$\Gamma (\Lambda^{n-1}A^*) \longrightarrow \mathfrak{X}(M),\,\,\rho (\alpha_1\wedge \cdots  \wedge \alpha_{n-1}) := [\alpha_1, \ldots ,\alpha_{n-1}, ~]$$
and hence, a vector bundle map $ \rho : \Lambda^{n-1}A^* \longrightarrow TM$. Rest of the proof is similar to the last part of the proof of Theorem \ref{classification-filippov-algeb} and we omit the details.
\end{proof}

\begin{remark}\label{Nambu-Gerstenhaber-n-bracket}
Given a weak Lie-Filippov algebroid pair $(A, A^*)$ of order $n$ over $M,$ we shall refer the induced $n$-ary bracket on  $\Gamma (\Lambda^\bullet A^*)$ as the induced Nambu-Gestenhaber bracket. 
\end{remark}
The following is one of the main results of this section.
\begin{thm}\label{NP-manifold-weak-Lie-algebroid-pair} 
Let (M, \{~, \ldots, ~\}) be a Nambu-Poisson manifold of order $n\geq 3$. Let $P\in \Gamma (\Lambda^nTM)$ be the associated Nambu-Poisson tensor. Then the pair $(TM, T^*M)$ is a weak Lie-Filippov algebroid pair of order $n\geq 3$ over $M$. 
\end{thm}

\begin{proof}
Recall from Equation (\ref{anchor-NP-manifold}) that $P$ induces a bundle mapping 
\begin{align*}
& P^\sharp: \Lambda^{n-1}T^\ast M\rightarrow TM, \,\,\big\langle \beta, P^\sharp (\alpha_1 \wedge \cdots \wedge \alpha_{n-1})\big\rangle = P (\alpha_1, \ldots , \alpha_{n-1}, \beta), 
\end{align*}
for all $\alpha_1, \ldots , \alpha_{n-1}, \beta \in \Omega^1(M).$

From Definition \ref{Nambu-form-bracket} we have the Nambu form-bracket on the space of $1$-forms as
\begin{align*}
[\alpha_1, \ldots, \alpha_n]
&= d(P(\alpha_1, \ldots, \alpha_n)) + \sum_{k=1}^n (-1)^{n+k}\iota_{P^\sharp(\alpha_1\wedge \cdots \wedge\widehat{\alpha}_k\wedge \cdots \wedge \alpha_n)}d\alpha_k\\
&=  \sum_{k=1}^n (-1)^{n+k} \mathcal L_{P^\sharp(\alpha_1\wedge \cdots \wedge \widehat{\alpha}_k\wedge \ldots \wedge \alpha_n)}\alpha_k-(n-1)d(P(\alpha_1, \ldots ,\alpha_n)),
\end{align*}
where $\alpha_k,~k=1, \ldots, n$ are $1$-forms on $M.$ On $TM$ we take the usual Lie algebroid structure and consider the vector bundle map (\ref{anchor-NP-manifold}) as the anchor. Then by properties $(1)$, $(3)$, of Proposition \ref{properties-Nambu-form-bracket} and by Remark \ref{fundamental-identity-NP-manifold}, all the defining conditions (cf. Definition \ref{weak-Lie-Filippov-algebroid-pair}) for the pair $(TM, T^*M)$ to be a weak Lie-Filippov algebroid pair of order $n\geq 3$ over $M$ hold modulo proving
$$ [P^\sharp (\alpha_1 \wedge \cdots \wedge \alpha_{n-1}), P^\sharp (\beta_1 \wedge \cdots \wedge \beta_{n-1})] = \sum_{i=1}^{n-1}P^\sharp (\beta_1 \wedge \cdots \wedge [\alpha_1, \ldots , \alpha_{n-1}, \beta_i] \wedge \cdots \wedge \beta_{n-1})$$ for all $d$-closed $1$-forms $\alpha_i,~ 1\leq i \leq n-1$ and for any $1$-forms $\beta_j,~ 1\leq j\leq n-1,$ on $M$.
To prove the above equality, we use the definition of the anchor $P^\sharp$ and Remark \ref{fundamental-identity-NP-manifold}.

If $\alpha_i = df_i$ and $\beta_j = dg_j,$ for $f_i, g_j \in C^\infty(M),$ are exact $1$-forms, then the above identity follows from Equation (\ref{expression-fundamental-id-hamiltonian}).

Next, replacing $dg_1$ by $fdg_1,$ where $f \in C^\infty(M)$ we get
\begin{align*}
& [P^\sharp(df_1 \wedge \cdots \wedge df_{n-1}), P^\sharp(f dg_1 \wedge \cdots \wedge dg_{n-1})]\\
=\,& [P^\sharp(df_1 \wedge \cdots \wedge df_{n-1}), fP^\sharp(dg_1 \wedge \cdots \wedge dg_{n-1})]\\
=\,& f[P^\sharp(df_1 \wedge \cdots \wedge df_{n-1}), P^\sharp(dg_1 \wedge \cdots \wedge dg_{n-1})]\\
&+ P^\sharp(df_1 \wedge \cdots \wedge df_{n-1})(f)P^\sharp(dg_1 \wedge \cdots \wedge dg_{n-1})\\
=\,& fP^\sharp([df_1, \ldots , df_{n-1}, dg_1]\wedge dg_2 \wedge  \cdots \wedge dg_{n-1})\\
&+ f\sum_{i=2}^{n-1}P^\sharp(dg_1 \wedge \cdots \wedge dg_{i-1} \wedge [df_1, \ldots , df_{n-1}, dg_i]\wedge \cdots \wedge dg_{n-1})\\
&+ P^\sharp(df_1 \wedge \cdots \wedge df_{n-1})(f)P^\sharp(dg_1 \wedge \cdots \wedge dg_{n-1})\\
=\,& P^\sharp([df_1, \ldots , df_{n-1}, fdg_1]\wedge dg_2 \wedge  \cdots \wedge dg_{n-1})\\
&+ \sum_{i=2}^{n-1}P^\sharp(fdg_1 \wedge \cdots \wedge dg_{i-1} \wedge [df_1, \ldots , df_{n-1}, dg_i]\wedge \cdots \wedge dg_{n-1}).
\end{align*}
Since closed forms are locally exact, we may conclude using the above observations that for all closed $1$-forms $\alpha_i,~ 1\leq i \leq n-1$ and for any $1$-forms $\beta_j,~ 1\leq j\leq n-1,$ on $M$ the equality
$$ [P^\sharp (\alpha_1 \wedge \cdots \wedge \alpha_{n-1}), P^\sharp (\beta_1 \wedge \cdots \wedge \beta_{n-1})] = \sum_{i=1}^{n-1}P^\sharp (\beta_1 \wedge \cdots \wedge [\alpha_1, \ldots , \alpha_{n-1}, \beta_i] \wedge \cdots \wedge \beta_{n-1})$$
holds. This completes the proof.
\end{proof} 

We have just seen that for a Nambu-Poisson manifold $M$ of order $n~(n\geq 3),$ the pair $(TM, T^*M)$ is a weak Lie-Filippov algebroid pair of order $n\geq 3$ over $M$. Therefore, by Theorem \ref{classification-weak-lie-filippov-algebroid-pair}, there is an $n$-ary bracket on the graded $C^\infty(M)$-module $\Gamma (\Lambda^\bullet T^*M)=\Omega^\bullet(M)$ which together with the wedge product makes it a Nambu-Gerstenhaber algebra of order $n$. We end this section proving a result that shows how this $n$-ary bracket behaves with respect to the exterior differential on $M$.

\begin{prop}\label{compatibility-d-n-gerstenhaber}
Let $(M, \{~, \ldots, ~\})$ be a Nambu-Poisson manifold of order $n$. Then for any $1$-forms $\alpha_1, \ldots , \alpha_n \in \Omega^1(M),$ we have
$$d[\alpha_1, \ldots, \alpha_n] = \sum_{i=1}^n [\alpha_1, \ldots, d\alpha_i, \ldots  \alpha_n].$$
\end{prop}

\begin{proof}
It is enough to prove the identity for $\alpha_i = f_idg_i,$ where $f_i, g_i \in C^\infty(M),~ 1\leq i \leq n.$
A direct computation using properties $(1)$ and $(3)$ of Proposition \ref{properties-Nambu-form-bracket} shows that
\begin{eqnarray} 
[f_1dg_1, \ldots ,f_ndg_n] & = & f_1f_2 \cdots f_n[dg_1, \ldots , dg_n]\nonumber \\
& + & \sum_{i=1}^n(-1)^{n-i} f_1\cdots \hat{f}_i \cdots f_n(P^\sharp (dg_1 \wedge \cdots \wedge \hat{dg}_i \wedge \cdots dg_n)f_i)dg_i,\nonumber 
\end{eqnarray}
for $f_1, \ldots, f_n,~ g_1, \ldots ,g_n \in C^\infty(M).$

One can write out $d[f_1dg_1, \ldots ,f_ndg_n]$ as a finite linear combination of terms of the form 
$df_i \wedge dg_j, (i\neq j)$, $df_i \wedge dg_i$, $dg_i\wedge dg_j$, $dg_j \wedge dg_i$, $df_i \wedge [dg_1, \ldots ,dg_n]$ and 
$[dg_1, \ldots, df_i, \ldots  ,dg_n]\wedge dg_i$. The coefficients of these terms are smooth functions on $M$. Similarly, one can expand
$$\sum_{i=1}^n [f_1dg_1, \ldots, d(f_idg_i), \ldots , f_ndg_n]$$
with respect to the same terms. To check that
$$d[f_1dg_1, \ldots ,f_ndg_n]= \sum_{i=1}^n [f_1dg_1, \ldots, d(f_idg_i), \ldots , f_ndg_n]$$
it suffices to match the coefficients on both quantities term-wise. 
We use property $(2)$ of Proposition \ref{properties-Nambu-form-bracket} and the fact that 
$$[df_1, \ldots , df_{n-1}, f_n] = P^\sharp(df_1 \wedge \cdots \wedge df_{n-1})f_n = \{f_1, \ldots, f_n\}$$
to observe that
\begin{enumerate}
\item the coefficients of $df_i \wedge dg_j, ~i\neq j,$ are the same in both the quantities;
\item the coefficients of $df_i \wedge dg_i$ are zero;
\item the coefficients of $dg_i\wedge dg_j$  are the same as the coefficients of $dg_j \wedge dg_i$ and hence, they cancel each other in one quantity while it is zero in the other;
\item the coefficients of $df_i \wedge [dg_1, \ldots ,dg_n]$ are the same in both the quantities;
\item the coefficients of $[dg_1, \ldots, df_i, \ldots  ,dg_n]\wedge dg_i$ are the same in both the quantities.
\end{enumerate}
This completes the proof.
\end{proof} 

For instance, let us look at the following classical example from \cite{nambu}.
\begin{exam}\label{example-nambu}
Consider $M = \mathbb{R}^3$ with the Nambu structure $P = \frac{\partial}{\partial x_1} \wedge \frac{\partial}{\partial x_2} \wedge \frac{\partial}{\partial x_3}$
of order $3.$ Then the Nambu form-bracket on the space of $1$-forms on $\mathbb{R}^3$ is explicitly given by the following formulas
\begin{align*}
& [ f_1 dx_1 , f_2 dx_1 , f_3 dx_1 ] =  0,~~[ g_1 dx_2 , g_2 dx_2 , g_3 dx_2 ] = 0,~~[ h_1 dx_3 , h_2 dx_3 , h_3 dx_3 ] = 0,\\ 
& [ f_1 dx_1 , f_2 dx_1 , g_1 dx_2 ] = g_1 ( f_2 \frac{\partial f_1}{\partial x_3}  - f_1 \frac{\partial f_2}{\partial x_3} ) dx_1,\\
& [ f_1 dx_1 , f_2 dx_1 , h_1 dx_3 ] = h_1 ( f_1 \frac{ \partial f_2}{\partial x_2}  -  f_2 \frac{\partial f_1}{\partial x_2} ) dx_1 ,\\
& [ f_1 dx_1 , g_1 dx_2 , g_2 dx_2 ] =  f_1 ( g_1 \frac{\partial g_2}{\partial x_3} - g_2 \frac{\partial g_1}{\partial x_3} ) dx_2,\\
& [ g_1 dx_2 , g_2 dx_2 , h_1 dx_3 ] = h_1 ( g_2 \frac{\partial g_1}{\partial x_1} - g_1  \frac{\partial g_2}{\partial x_1} ) dx_2 ,\\
& [ f_1 dx_1 , h_1 dx_3 , h_2 dx_3 ] = f_1 ( h_2 \frac{\partial h_1}{\partial x_2} - h_1  \frac{\partial h_2}{\partial x_2} ) dx_3,\\
& [ g_1 dx_2 , h_1 dx_3 , h_2 dx_3 ] = g_1 ( h_1 \frac{\partial h_2}{\partial x_1}  -  h_2  \frac{\partial h_1}{\partial x_1} ) dx_3,\\
& [ f_1 dx_1 , g_1 dx_2 , h_1 dx_3 ] = g_1h_1  \frac{\partial f_1}{\partial x_1} dx_1 + f_1h_1  \frac{\partial g_1}{\partial x_2} dx_2 + f_1g_1  \frac{\partial h_1}{\partial dx_3} dx_3 ,
\end{align*}
for $ f_1, f_2 , f_3 , g_1 , g_2 , g_3 , h_1 , h_2 , h_3 \in C^\infty(\mathbb{R}^3)$.

For this Nambu form-bracket one can directly check that the compatibility condition
$$ d [\alpha_1 , \alpha_2 , \alpha_3 ] =  [d \alpha_1 , \alpha_2 , \alpha_3] + [\alpha_1 , d \alpha_2 , \alpha_3] + [\alpha_1 , \alpha_2 , d\alpha_3 ]$$
holds, for $\alpha_1, \alpha_2 , \alpha_3 \in \Omega^1(\mathbb{R}^3)$ in each of the above cases.
\end{exam}

\begin{remark}\label{Nambu-Lie-example}
Nambu structures on Lie algebroids are further generalizations of Nambu-Poisson manifold \cite{wade}. Let $A\rightarrow M$ be a Lie algebroid. Assume that the vector bundle rank of  $A$ is $m.$ Let $d_A$ denote the coboundary operator of the Lie algebroid cohomology complex of $A$. Recall that a smooth section $\Pi \in \Gamma(\Lambda^n A) $ is said to be a {\it{Nambu structure}} of order $n$ $(n\in\mathbb{N},~3\leq n \leq m)$ on $A$ if 
$$ [\Pi^\sharp(\alpha), \Pi]^\sharp(\beta) = -\Pi^\sharp (\iota_{\Pi^\sharp\beta}d_A\alpha)~~~ \mbox{for any}~~~ \alpha, \beta  \in \Gamma(\Lambda^{n-1} A^\ast),$$ where $[~, ~]$ is the Schouten bracket on the graded commutative algebra $\Gamma (\Lambda^{\bullet} A),$ $\iota$ is the contraction operator and $\Pi^\sharp : \Lambda^{n-1} A^\ast \rightarrow A$ is the bundle map induced by $\Pi.$ 

It may be remarked that if $A$ is a Lie algebroid  equipped with  a Nambu structure of order $n$ such that every $d_A$-closed section  $M\longrightarrow A^*$ is locally $d_A$-exact, then the proof of Theorem \ref{NP-manifold-weak-Lie-algebroid-pair} may be adapted to show that the pair $(A, A^*)$ is weak Lie-Filippov algebroid pair of order $n$ over $M$.
\end{remark}

\section{Weak Lie-Filippov bialgebroid}\label{$6$}
The notion of a Lie bialgebroid was introduced by Mackenzie and Xu \cite{mac-xu}, as a generalization of both Poisson manifold and Lie bialgebra. Lie bialgebroids are the infinitesimal form of Poisson groupoids. 

Recall that a Lie bialgebroid is a pair $(A, A^\ast)$ of Lie algebroids in duality, satisfying the compatibility condition 
$$d_\ast[X, Y] = [d_\ast X, Y] + [X , d_\ast Y],$$ for all $X,~ Y \in \Gamma A.$
Here $d_\ast$ is the differential on $\Gamma (\Lambda^\bullet A)$ defined by the Lie algebroid structure of $A^\ast$ and $[~, ~]$ is the Gerstenhaber bracket  on $\Gamma (\Lambda^\bullet A)$ defined by the Lie algebroid structure on $A.$ Note that this condition is equivalent to the condition $(16)$ in \cite{mac-xu}.

Let $d_A$ be the coboundary operator in the Lie algebroid complex $\Gamma (\Lambda^\bullet A^\ast)$ of the Lie algebroid $A$ and $[~, ~]_\ast$ be the Gerstenhaber bracket in $\Gamma (\Lambda^\bullet A^\ast)$ extending the Lie bracket of $\Gamma A^\ast.$ Then the compatibility condition may be stated equivalently as 
$$d_A[\alpha, \beta]_\ast = [d_A\alpha, \beta]_\ast + [\alpha , d_A\beta]_\ast,$$ for all $\alpha,~ \beta \in \Gamma A^\ast$ (\cite{kosmann}, \cite{mac-xu}). 

The aim of this section is to address the question raised in the introduction about the existence of some notion of bialgebroid associated to Nambu-Poisson manifold of order $n>2$. To answer this, we introduce a notion which we call {\it weak Lie-Filippov bialgebroid} and show that such objects arise from Nambu-Poisson manifolds. Roughly speaking, a weak Lie-Filippov bialgebroid of order $n$ is a pair $(A, A^\ast)$ consisting of a Lie algebroid $A$ over a smooth manifold $M$ such that the space of sections $\Gamma A^*$ of the dual bundle $A^\ast$ admits an $n$-ary bracket making  $(A, A^*)$ a weak Lie-Filippov algebroid pair of order $n$ over $M$ and the Lie algebroid structure on $A$ and the induced Nambu-Gerstenhaber bracket on $\Gamma (\Lambda^\bullet A^*)$ are related by some suitable compatibility condition.  This compatibility condition may be viewed as a generalization of the equivalent compatibility condition for Lie bialgebroids as stated above. Finally, we prove some results which are consequences of this structure.

We begin with a motivating example. 
\begin{exam}\label{motivating-example}
Let $M$ be a Nambu-Poisson manifold of order $n$. Then by Theorem \ref{NP-manifold-weak-Lie-algebroid-pair}, in the previous section, we know that the pair $(TM, T^\ast M)$ is a weak Lie-Filippov algebroid pair of order $n$ over $M$. Moreover, by Proposition \ref{compatibility-d-n-gerstenhaber}, the induced Nambu-Gerstenhaber bracket (cf. Remark \ref {Nambu-Gerstenhaber-n-bracket}) on $\Omega^\bullet(M)$ satisfies
$$d[\alpha_1, \ldots, \alpha_n] = \sum_{i=1}^n [\alpha_1, \ldots, d\alpha_i, \ldots  \alpha_n],~ \alpha_i \in \Omega^1(M).$$
\end{exam}

\begin{defn}\label{defn-weak-lie-filippov-bialgebroid}
Let $A$ be a smooth vector bundle over a smooth manifold $M$. Then the pair $(A, A^\ast )$ is said to be a {\it weak Lie-Filippov bialgebroid} of order $n,$ $(n \geq 3$) over $M$, if 
\begin{enumerate}
\item $(A, A^*)$ is a weak Lie-Filippov algebroid pair of order $n$ over $M$;
\item the following compatibility condition holds.
$$ d_A[\alpha_1, \ldots , \alpha_n] = \sum_{i=1}^n [\alpha_1, \ldots , d_A\alpha_i, \ldots , \alpha_n],$$ for any $\alpha_i \in \Gamma A^\ast,$ $i \in \{1, \ldots , n\}$,  where $[~, \ldots , ~]$ is the  Nambu-Gestenhaber bracket on the graded algebra $\Gamma(\Lambda^\bullet A^\ast)$ corresponding to the weak Lie-Filippov algebroid pair $(A,A^\ast)$ over $M$.
\end{enumerate}
The anchor $\rho$ of the weak Lie-Filippov algebroid pair is called the anchor of the weak Lie-Filippov bialgebroid $(A, A^*).$ A weak Lie-Filippov bialgebra of order $n$ is a weak Lie-Filippov bialgebroid of order $n$ over a point.
\end{defn} 
\begin{remark}\label{remark-terminology}
Ideally, like the Poisson case, one might expect that for a Nambu-Poisson manifold $M$ of order $n,$ the pair $(TM, T^\ast M)$ would form a Lie-Filippov bialgebroid of order $n$ (a notion obtained as above by dropping the restriction `$d_A$-closed' in the defining conditions of Definition \ref{weak-Lie-Filippov-algebroid-pair}). Unfortunately, as remarked before (Remark \ref{terminology-weak}), due to the rigidity of Nambu structure, we end up getting a weaker notion.
\end{remark}
From Theorem \ref{NP-manifold-weak-Lie-algebroid-pair} and Proposition \ref{compatibility-d-n-gerstenhaber}, we deduce

\begin{corollary}\label{NP-manifold-weak-lie-filippov-bialgebroid}
Suppose $M$ is a Nambu-Poisson manifold of order $n,~ (n\geq 3).$ Then the pair $(TM, T^\ast M)$ is a weak Lie-Filippov bialgebroid of order $n$ over $M$.
\end{corollary}
The above fact has been illustrated by Example \ref{example-nambu}.

Next, we prove a result analogous to Theorem \ref{classification-weak-lie-filippov-algebroid-pair} for weak Lie-Filippov bialgebroids.  

Recall the following definitions from \cite{xu}.

\begin{defn}\label{strong-diff-gerstenhaber-algebra}
A {\it differential Gerstenhaber algebra} is a Gerstenhaber algebra 
$$(\mathcal A =\oplus_{i\in \mathbb Z}\mathcal A^i, \wedge, [~,~])$$ equipped with a differential $\delta$, which is a derivation of degree $1$ with respect to $\wedge$ and $\delta^2 = 0.$  It is called a {\it strong differential Gerstenhaber algebra} if, in addition, $\delta$ is a derivation of graded Lie bracket (cf. Definition \ref{gerstenhaber-algebra}).
\end{defn}

Given a smooth vector bundle $A$ over $M$, there is a one-to-one correspondence between Lie bialgebroid pairs $(A,A^\ast)$ and strong differential Gerstenhaber algebra structures on $\Gamma(\Lambda^\bullet A)$ (\cite{xu}). Note that Lie bialgebroid is a dual concept, in the sense that if $(A,A^\ast)$ is a Lie bialgebroid so is $(A^\ast, A)$ (\cite{mac-xu,kosmann}). Hence, there is a one-to-one correspondence between Lie bialgebroids $(A,A^\ast)$  and strong differential Gerstenhaber structures on $\Gamma(\Lambda^\bullet A^\ast).$ 

We prove an $n$-ary version of the above result by introducing a notion of {\it strong Nambu-Gerstenhaber algebra of order $n$}. 

\begin{defn}\label{strong-Nambu-Gerstenhaber-n-algebra}
A Nambu-Gerstenhaber algebra $(\mathcal A, \wedge, [~, \ldots , ~], \delta)$ of order $n$ is said to be a {\it strong Nambu-Gerstenhaber algebra of order $n$} if the differential $\delta$ on $\mathcal A$, is a derivation of the graded Nambu-Gerstenhaber $n$-ary bracket on $\mathcal A$.
\end{defn}

The following result is immediate from Theorem \ref{classification-weak-lie-filippov-algebroid-pair} and Definitions \ref{defn-weak-lie-filippov-bialgebroid}, \ref{strong-Nambu-Gerstenhaber-n-algebra}.

\begin{prop}\label{classification-weak-lie-filippov-strong-nambu-gerstenhaber}
Let $A$ be a smooth vector bundle over a smooth manifold $M$. Then $(A, A^\ast)$ is a weak Lie-Filippov bialgebroid of order $n$ if and only if $(\Gamma(\Lambda^\bullet A^\ast), \wedge)$ is a strong Nambu-Gerstenhaber algebra of order $n$. 
\end{prop}

In \cite{mac-xu,kosmann}, the authors showed that if $A$ is a smooth vector bundle over a smooth manifold $M$ such that $(A, A^\ast)$ is a Lie bialgebroid, then there is a canonical Poisson structure on the base manifold $M$. It is interesting to investigate this aspect in the case when $(A, A^\ast)$ is a weak Lie-Filippov bialgebroid. To this end, we need the following useful lemma.

\begin{lemma}\label{compatibility-f}
Let $(A, A^\ast)$ be a weak Lie-Filippov bialgebroid of order $n$ over a smooth manifold $M$. Let $[~, \ldots ,~]$ be the induced Nambu-Gerstenhaber $n$-ary bracket on $\Gamma (\Lambda^\bullet A^\ast).$ For any $\alpha_i \in \Gamma A^\ast,~ i\in \{1, \ldots , n-1\}$ and $f \in C^\infty(M)$, we have
$$ d_A [ \alpha_1, \ldots ,\alpha_{n-1}, f] = \sum_{i=1}^{n-1}[\alpha_1, \ldots, d_A\alpha_i, \ldots , \alpha_{n-1}, f] + [\alpha_1, \ldots , \alpha_{n-1}, d_Af].$$
\end{lemma}  

\begin{proof}
Let $\alpha_n \in \Gamma A^\ast.$  Note that 
$$ d_A [ \alpha_1, \ldots , f\alpha_n] =  \sum_{i=1}^{n-1}[\alpha_1, \ldots,d_A\alpha_i, \ldots , f\alpha_n]+ [\alpha_1, \ldots , \alpha_{n-1}, d_A(f\alpha_n)].$$
We simplify the left hand side as
\begin{align*}
d_A [ \alpha_1, \ldots , f\alpha_n] 
&=  d_A([\alpha_1, \ldots , f] \wedge \alpha_n + f [\alpha_1, \ldots , \alpha_n])\\
&= d_A[\alpha_1, \ldots , \alpha_{n-1}, f] \wedge \alpha_n + [\alpha_1, \ldots , \alpha_{n-1}, f] d_A \alpha_n\\ 
&+ d_Af \wedge [\alpha_1, \ldots , \alpha_n] + f d_A[\alpha_1, \ldots , \alpha_n].
\end{align*}
On the other hand, the right hand side can be written as
\begin{align*}
& \sum_{i=1}^{n-1}[\alpha_1, \ldots,d_A\alpha_i, \ldots , f\alpha_n]+ [\alpha_1, \ldots , \alpha_{n-1}, d_A(f\alpha_n)]\\
&= \sum_{i=1}^{n-1}[\alpha_1, \ldots,d_A\alpha_i, \ldots , f]\wedge \alpha_n + \sum_{i=1}^{n-1}f[\alpha_1, \ldots,d_A\alpha_i, \ldots , \alpha_n]\\
&+ [\alpha_1, \ldots , \alpha_{n-1}, d_Af] \wedge \alpha_n + d_Af \wedge [\alpha_1, \ldots , \alpha_n]\\
&+ [\alpha_1, \ldots , \alpha_{n-1}, f]d_A\alpha_n + f[\alpha_1, \ldots , d_A\alpha_n].
\end{align*}

Comparing right hand sides of the above equalities we deduce
$$d_A[\alpha_1, \ldots , \alpha_{n-1}, f] \wedge \alpha_n= (\sum_{i=1}^{n-1}[\alpha_1, \ldots,d_A\alpha_i, \ldots , f]+ [\alpha_1, \ldots , \alpha_{n-1}, d_Af]) \wedge \alpha_n.$$ 
Thus, the result follows as $\alpha_n$ is arbitrary. 
\end{proof} 

In order to state the next theorem we need to formulate a notion of morphisms between weak Lie-Filippov bialgebroids. 
\begin{defn}\label{definition-bialgebroid-morphism}
Suppose $(A,A^\ast)$ and $(B,B^\ast)$ are two weak Lie-Filippov bialgebroids of order $n$ over $M$. A weak Lie-Filippov bialgebroid morphism $\phi : (A,A^\ast) \to (B,B^\ast)$ is a morphism $\phi:A\to B$ of Lie algebroids so that $\phi^\ast : B^\ast \to A^\ast$ preserves the $n$-ary brackets on the respective spaces of smooth sections and commutes with the anchors of $(A,A^\ast)$ and $(B,B^\ast).$  
\end{defn}

\begin{thm}\label{weak-lie-filippov-nambu-poisson-on base}
Let $(A, A^\ast)$ be a weak Lie-Filippov bialgebroid of order $n$ over a smooth \linebreak manifold $M$ for $n\geq 3$. Then there is a canonical Nambu-Poisson structure of order $n$ on $M$ such that the anchor $a : A \to TM$ of the Lie algebroid $A$ is a morphism of weak Lie-Filippov bialgebroids $(A, A^*) \to (TM, T^*M).$ Moreover, if there is a morphism $(A,A^\ast) \to (B,B^\ast)$ of weak Lie-Filippov bialgebroids over $M,$ then the corresponding induced Nambu-Poisson structures on $M$ are the same. 
\end{thm} 
\begin{proof}
Define an $n$-bracket on $C^\infty(M)$ by
$$\{f_1, \ldots , f_n\} := \rho (d_Af_1 \wedge \cdots \wedge d_Af_{n-1})f_n = [d_Af_1, \ldots , d_Af_{n-1}, f_n],$$
where $d_A$ is the coboundary operator of the cochain complex $\Gamma ( \Lambda^\bullet A^\ast)$, $\rho$ is the anchor of the weak Lie-Filippov bialgebroid  $(A, A^\ast)$ and $[~, \ldots , ~]$ is the induced Nambu-Gerstenhaber bracket on $\Gamma (\Lambda^\bullet A^\ast).$

From Lemma \ref{compatibility-f}, it follows that 
\begin{eqnarray}\label{eq-5.1}
d_A\{f_1, \ldots , f_n\} & = & [d_Af_1, \ldots , d_Af_n].
\end{eqnarray}
Clearly, the bracket is skew-symmetric in the first $(n-1)$ entries. Thus, to prove the skew-symmetric property of the above defined bracket, it is enough to prove 
$$\{f_1, \ldots , f_{n-2}, f_{n-1}, f_n\}= -\{f_1, \ldots , f_{n-2}, f_n, f_{n-1}\}.$$ The last assertion is equivalent to proving $\{f_1, \ldots , f_{n-2}, f, f\}= 0,$ for any $f_1, \ldots,  f_{n-2}, f \in C^\infty(M),$ which we check below.
Note that for any $f_1, \ldots,  f_{n-2}, f \in C^\infty(M),$
\begin{eqnarray}\label{eq-5.2}
d_A\{f_1, \ldots , f_{n-2}, f, f\} & = & [d_Af_1, \ldots , d_Af_{n-2}, d_Af, d_Af] = 0.  
\end{eqnarray}
In particular, replacing $f$ by $f^2$ we get 
$$d_A\{f_1, \ldots , f_{n-2}, f^2, f^2\} = 0.$$ By definition of the bracket we obtain $d_A(\rho ( d_Af_1 \wedge \cdots \wedge d_Af^2)(f^2)) = 0.$ 
Now, we use the derivation property of $d_A$ and that of a vector field to deduce 
$$d_A (f\rho ( d_Af_1 \wedge \cdots \wedge d_Af)(f^2)) = 0, ~~~d_A (f^2\rho ( d_Af_1 \wedge \cdots \wedge d_Af)(f)) = 0.$$ 
The last equality implies $d_A (f^2)\rho (d_Af_1 \wedge \cdots \wedge d_Af)(f) + f^2 d_A(\rho (d_Af_1 \wedge \cdots \wedge d_Af)(f)) = 0.$\\
By definition of the $n$-ary bracket and from equation \eqref{eq-5.2}, we see that 
$$d_A(\rho ( d_Af_1 \wedge \cdots \wedge d_Af)(f)) = 0.$$ 
Therefore, we get $\rho ( d_Af_1 \wedge \cdots \wedge d_Af)(f)d_A(f^2) = 0$ or $f\rho ( d_Af_1 \wedge \cdots \wedge d_Af)(f)d_Af = 0.$\\  
Taking wedge product with $d_Af_1 \wedge \cdots \wedge d_Af_{n-2}$, the last equality yields 
$$f\rho ( d_Af_1 \wedge \cdots \wedge d_Af)(f)d_Af_1 \wedge \cdots \wedge d_Af_{n-2} \wedge d_Af = 0.$$ 
This means, $f\rho ( d_Af_1 \wedge \cdots \wedge d_Af)(f) \rho (d_Af_1 \wedge \cdots \wedge d_Af_{n-2} \wedge d_Af)(f) = 0.$ In other words, $f(\rho ( d_Af_1 \wedge \cdots \wedge d_Af)(f))^2 = 0.$ Let $u=(\rho ( d_Af_1 \wedge \cdots \wedge d_Af)(f))^2$. Thus, we have either $f(x)=0$ or $u(x)=0$ for every $x$. Note that $u(x)=0$ if $f(x)=0$ in an open neighborhood of $x$ or if $f(x)\neq 0$. Otherwise, $f(x)=0$ and $f(x_n)\neq0$ for a sequence $x_n \to x$ as $n\to \infty$. It follows that $u(x_n)=0$ and hence, $u(x)=0$ by continuity. Therefore, in all cases
$$ \rho ( d_Af_1 \wedge \cdots \wedge d_Af)(f) = \{f_1, \ldots , f_{n-2}, f, f\} = 0.$$
 
From the definition of the bracket, it is clear that it satisfies the derivation property. We prove the fundamental identity for the bracket using Hamiltonian formalism. For $f_1, \ldots , f_{n-1} \in C^\infty(M),$  define a vector field $X_{f_1 \ldots  f_{n-1}}$ by
$$ X_{f_1 \ldots f_{n-1}} = \rho (d_Af_1 \wedge \cdots \wedge d_Af_{n-1}).$$
Then for $g \in C^\infty(M),$
$$X_{f_1 \ldots  f_{n-1}}(g) = \rho (d_Af_1 \wedge \cdots \wedge d_Af_{n-1})(g) = \{f_1, \ldots , f_{n-1}, g\}.$$
Note that for  $f_1, \ldots , f_{n-1} \in C^\infty(M),$ and  $g_1, \ldots , g_{n-1} \in C^\infty(M),$
\begin{align*}
[X_{f_1 \ldots  f_{n-1}}, X_{g_1 \ldots  g_{n-1}}] 
& = [\rho (d_Af_1 \wedge \cdots \wedge d_Af_{n-1}), \rho (d_Ag_1 \wedge \cdots \wedge d_Ag_{n-1})]\\
& = \sum_{i=1}^{n-1}\rho (d_Ag_1 \wedge \cdots \wedge [d_Af_1, \ldots, d_Af_{n-1}, d_Ag_i] \wedge \cdots \wedge d_Ag_{n-1})\\
& = \sum_{i=1}^{n-1}\rho (d_Ag_1 \wedge \cdots \wedge d_A\{f_1, \ldots, f_{n-1}, g_i\} \wedge \cdots \wedge d_Ag_{n-1})\\
& = \sum_{i=1}^{n-1}X_{g_1 \ldots  \{f_1, \ldots, f_{n-1}, g_i\} \ldots  g_{n-1}}.
\end{align*}
Hence,
\begin{align*}
& \{f_1, \ldots , f_{n-1}, \{g_1, \ldots, g_n\}\} - \sum_{i=1}^n\{g_1, \ldots, \{f_1, \ldots , f_{n-1}, g_i\}, \ldots , g_n\}\\
& = X_{f_1 \ldots  f_{n-1}} X_{g_1 \ldots  g_{n-1}}(g_n) - X_{g_1 \ldots  g_{n-1}}X_{f_1 \ldots  f_{n-1}}(g_n) - \sum_{i=1}^{n-1} X_{g_1 \ldots  \{f_1, \ldots, f_{n-1}, g_i\} \ldots  g_{n-1}}(g_n)\\
& = [X_{f_1 \ldots  f_{n-1}}, X_{g_1 \ldots  g_{n-1}}](g_n) - \sum_{i=1}^{n-1} X_{g_1 \ldots  \{f_1, \ldots, f_{n-1}, g_i\} \ldots  g_{n-1}}(g_n) = 0.
\end{align*}
Therefore, we conclude that $M$ has a Nambu-Poisson structure of order $n$. Next we verify that the anchor $a:A\to TM$ of the Lie algebroid $A$ is, in fact, a morphism $(A, A^*) \to (TM, T^* M)$ of weak Lie-Filippov bialgebroids. Note that $a^\ast(df)=d_Af$. The anchor for any Lie algebroid is a Lie algebroid morphism, thus, it suffices to prove that for $1$-forms $\alpha_i$
$$a^\ast [\alpha_1,\ldots,\alpha_n]= [a^\ast \alpha_1,\ldots, a^\ast \alpha_n]$$ and
$$ \rho (a^\ast(\alpha_1\wedge \cdots \wedge \alpha_{n-1})) = P^\sharp (\alpha_1\wedge \cdots \wedge \alpha_{n-1}), $$ where $P$ is the Nambu-Poisson tensor on $M$  induced from the weak Lie-Filippov bialgebroid $(A, A^*).$  
For the second equation note that the anchor maps are vector bundle morphisms and hence, $C^\infty(M)$-linear. Therefore, it suffices to prove the equation for $\alpha_i=df_i$ for which 
\begin{align*}
& \rho (a^\ast(df_1\wedge \cdots \wedge df_{n-1}))(f_n) =\rho (d_Af_1\wedge \cdots \wedge d_Af_{n-1})(f_n)\\
&= \{f_1,\ldots,f_n\} = P^\sharp (df_1\wedge \ldots \wedge df_{n-1})(f_n).
\end{align*}
For the first equation, we may write each $\alpha_i$ as a linear combination of $f_idg_i$, and, expand both sides by linearity and the derivation property of the $n$-bracket. Note that $a^\ast$ is $C^\infty(M)$-linear and thus, it suffices to verify the equation for $\alpha_i =df_i$. We have 
\begin{displaymath}
a^\ast [df_1,\ldots,df_n] =a^\ast d \{f_1,\ldots,f_n\} =d_A \{f_1,\ldots,f_n\} = [ d_Af_1,\ldots, d_Af_n] =[a^\ast df_1,\ldots, a^\ast df_n].
\end{displaymath}

It remains to show that the induced Nambu-Poisson structure is independent up to morphisms of weak Lie-Filippov bialgebroids. Suppose that there is a morphism of weak Lie-Filippov bialgebroids $\phi: (A,A^\ast) \to (B,B^\ast)$ over $M$. Let us denote by $a_A,$ the anchor for any Lie algebroid $A$ and by $\rho_{A^\ast},$ the anchor for any weak Lie-Filippov bialgebroid $(A, A^\ast).$  It implies the existence of a commutative diagram of weak Lie-Filippov bialgebroids 
$$ \xymatrix{ (A,A^\ast) \ar[r]^{\phi} \ar[rd]_{a_B \circ \phi = a_A} & (B,B^\ast) \ar[d]^{a_B} \\
& (TM,T^\ast M)}$$
with an induced diagram of vector bundles
$$ \xymatrix{ T^\ast M \ar[rd]_{a_A^\ast} \ar[r]^{a_B^*} & B^\ast \ar[d]^{\phi^*} \\
& A^\ast}$$
Let $\{,\ldots,\}_{(A, A^*)}$ (respectively,  $\{,\ldots,\}_{(B,B^*)}$) be the Nambu-Poisson structures induced from $(A, A^*)$ (respectively, $(B, B^*)$). Using the commutative diagrams above we get
\begin{align*}
\{f_1,\ldots,f_n\}_{(A, A^*)}
& = \rho_{A^\ast}(d_Af_1\wedge  \cdots \wedge d_Af_{n-1})(f_n) \\
&= \rho_{A^\ast}(a_A^\ast df_1 \wedge  \cdots \wedge a_A^\ast df_{n-1})(f_n)\\
&= \rho_{A^\ast}(\phi^\ast a_B^\ast df_1 \wedge  \cdots \wedge \phi^\ast a_B^\ast df_{n-1})(f_n)\\
&= \rho_{A^\ast}\phi^\ast (a_B^\ast df_1 \wedge  \cdots \wedge a_B^\ast df_{n-1})(f_n)\\
&= \rho_{B^\ast}(d_Bf_1 \wedge  \cdots \wedge  d_Bf_{n-1})(f_n)\\
&=\{f_1,\ldots,f_n\}_{(B, B^*)} .
\end{align*}
This completes the proof of the theorem. 
\end{proof}

\begin{remark}
Observe that for a Nambu-Poisson manifold $M$ of order $n$, if we consider the weak Lie-Filippov bialgebroid $(TM, T^\ast M)$ of Example \ref{motivating-example}, then its induced Nambu-Poisson structure on $M$ coincides with the given Nambu-Poisson structure on $M$.
\end{remark}
\vspace*{0.5cm}

\section{Examples from Nambu-Lie group}\label{$7$}
It is well known that Lie bialgebroids are the infinitesimal form of Poisson groupoids \cite{mac}, just like, Lie bialgebras are the infinitesimal form of Poisson-Lie groups \cite{lu-weinstein}. A Poisson-Lie group is a Lie group together with a multiplicative Poisson structure. Following the Poisson-Lie case, in \cite{vaisman}, the author introduced the notion of a Nambu-Lie group as a Lie group endowed with a multiplicative Nambu structure. More precisely, a Lie group $G$ is a Nambu-Lie group if it admits a Nambu structure $P \in \Gamma(\Lambda^nTG)$ of order $n$ satisfying
$$P(gh) = (l_g)_*P(h) + (r_h)_*P(g),$$ where $l_g$ and $r_h$ denote the left translation by $g$ and right translation by $h$ in $G$, respectively. The multiplicativity of $P$ implies $P(e)= 0,$ where $e$ is the identity element of $G$. 

Recall that \cite{gra-mar,vaisman} the {\it intrinsic derivative} of $P$ at $e,$ is the linear map
$$\delta = d_eP : \mathfrak g \longrightarrow \Lambda^n \mathfrak g$$ defined by
$$ X\mapsto (\mathcal L_{\overline{X}}P)(e)$$ where $\mathfrak g$ is the Lie algebra of $G$ and $\overline{X}$ is any vector field whose value at $e$ is $X$. Then by Theorem $(2.2)$, \cite{vaisman}, $\delta$ is a $1$-cocycle of $\mathfrak g$ with coefficients in the adjoint representation on $\Lambda ^n\mathfrak g.$ Moreover, the dual $\delta^*$ of $\delta$ defines an $n$-ary operation 
$$[~, \ldots, ~] = \delta^* : \Lambda^n \mathfrak g^* \longrightarrow \mathfrak g^*$$ on $\mathfrak g^*$  which makes it a Filippov algebra of order $n.$ 

Recall \cite {ciccoli} that a Lie-Filippov bialgebra of order $n$ consists of a pair $(\mathfrak g, \mathfrak g^\ast)$ where $\mathfrak g$ is a Lie algebra and $\mathfrak g^\ast$ is a Filippov algebra of order $n$ such that the map $\delta : \mathfrak g \rightarrow \Lambda^n \mathfrak g$ obtained by dualizing the Filippov bracket on $\mathfrak g^\ast$ is a Lie algebra $1$-cocycle of $\mathfrak g$ with  coefficients in the adjoint representation on $\Lambda^n \mathfrak g.$ 
Thus, for a Nambu-Lie group $(G, P)$ of order $n$ with Lie algebra $\mathfrak g,$ the pair $(\mathfrak g , \mathfrak g^\ast)$ is a Lie-Filippov bialgebra of order $n,$ called the tangent Lie-Filippov bialgebra of $G$.

Explicitly, the Filippov $n$-bracket $[~, \ldots, ~]$ on $\mathfrak{g}^*$ is given by the Nambu form-bracket as
$$ [\alpha_1, \ldots, \alpha_n] = [\widetilde{\alpha_1}, \ldots, \widetilde{\alpha_n}] (e),$$
for $\alpha_i \in \mathfrak{g}^*$ and $\widetilde{\alpha_i}$'s are left invariant $1$-forms on $G$ with $\widetilde{\alpha_i} (e) = \alpha_i$, for all
$i = 1, \ldots, n$ \cite{vaisman}.

Let us look at the following example.
\begin{exam}\label{example-heisenberg}
Consider the Heisenberg group
$$   H (1,1) = \{ \left( \begin{array}{ccc}
1 & x & z\\ 
0 & 1 & y \\
0 & 0 & 1
\end{array}
\right) |~  x, y, z \in \mathbb{R} \} $$
with the matrix multiplication as the group operation. Then the $3$-tensor $P = y \frac{\partial}{\partial x}  \wedge \frac{\partial}{ \partial z}  \wedge \frac{\partial}{\partial y}$
defines a Nambu-Lie group structure on $H (1,1)$. The left invariant $1$-forms on $H(1,1)$ are $dx,~ dy,~ dz - x dy$ \cite{vaisman}. Hence, the dual $\mathfrak{g}^*$
of the Lie algebra $\mathfrak{g}$ of the Heisenberg group is generated by the above left invariant $1$-forms evaluated at the identity. Moreover, the Filippov
algebra bracket on $\mathfrak{g}^*$ is given by
\begin{align*}
 [dx, dy, dz - x dy] =& [dx, dy, dz] - [dx, dy, x dy] \\
=& d \{x, y, z \} - x ~ d \{x, y, y\} - P^\sharp (dx \wedge dy)(x) dy \\
=& - dy + y \frac{\partial x}{\partial z} dy = - dy
\end{align*}
and zero otherwise.
\end{exam}

The following result is a characterization of Lie-Filippov bialgebra. 
\begin{prop}
Let $\mathfrak g$ be a Lie algebra such that $\mathfrak g^\ast$ admits an $n$-ary bracket $[~, \ldots, ~]$ making it a Filippov algebra. Then the pair  $(\mathfrak g , \mathfrak g^\ast)$ is a Lie-Filippov bialgebra of order $n$ if and only if for any $\alpha_1, \ldots, \alpha_n \in \mathfrak g^*$  
$$d_{\mathfrak g}([\alpha_1, \ldots,  \alpha_n]) = \sum_{i=1}^n [\alpha_1, \ldots, d_{\mathfrak g}\alpha_i, \ldots,  \alpha_n],$$ where the bracket on the right hand side is the extended $n$-Gerstenhaber algebra bracket on $\Lambda^\bullet \mathfrak g^*$ (cf. Example \ref{lie-n-gerstenhaber}) induced from the Filippov algebra structure on $\mathfrak{g}^*,$ $d_{\mathfrak g}$ being the differential of the Chevalley-Eilenberg complex  $\Lambda^n\mathfrak g^\ast.$
\end{prop}

\begin{proof}
Assume that the pair $(\mathfrak g, \mathfrak g^\ast)$ is a Lie-Filippov bialgebra of order $n.$  Let $\delta$ be the dual of 
$$[~,\ldots, ~]: \underbrace{\mathfrak g^\ast \wedge \cdots \wedge \mathfrak g^\ast}_{n\,\textup{copies}} \longrightarrow \mathfrak g^\ast.$$
Observe that the cocycle condition of $\delta$ may be reformulated as follows. For any $\alpha_1, \ldots, \alpha_n\in \mathfrak g^*$ and $X, Y \in \mathfrak g, $
$$\langle [\alpha_1, \ldots, \alpha_n], [X, Y]\rangle = \sum_{i=1}^n [\langle [\alpha_1, \ldots, ad_Y^*\alpha_i, \ldots,  \alpha_n], X\rangle -\langle [\alpha_1, \ldots, ad_X^*\alpha_i, \ldots,  \alpha_n], Y\rangle].$$
Therefore,
\begin{align*}
& d_{\mathfrak g}([\alpha_1, \ldots,  \alpha_n])(X, Y)\\
& = -\langle [\alpha_1, \ldots, \alpha_n], [X, Y] \rangle\\
& = \sum_{i=1}^n [\langle [\alpha_1, \ldots, ad_X^*\alpha_i, \ldots,  \alpha_n], Y\rangle -\langle [\alpha_1, \ldots, ad_Y^*\alpha_i, \ldots,  \alpha_n], X\rangle]\\
& = \sum_{i=1}^n [\langle [\alpha_1, \ldots, \iota_Xd_{\mathfrak g}\alpha_i, \ldots,  \alpha_n], Y\rangle -\langle [\alpha_1, \ldots, \iota_Yd_{\mathfrak g}\alpha_i, \ldots,  \alpha_n], X\rangle]
\end{align*}
since $ad^*_X\alpha _i = \iota_Xd_{\mathfrak g}\alpha_i$ for any $X \in \mathfrak g.$ To complete the proof, we need to show that the last equality is the same as
$$\sum_{i=1}^n [\alpha_1, \ldots, d_{\mathfrak g}\alpha_i, \ldots, \alpha_n](X, Y).$$ It is enough to establish this for $d_{\mathfrak g}\alpha_i$ of the form $\beta_i\wedge \gamma_i \in \Lambda^2\mathfrak g^*.$ Assuming $d_{\mathfrak g}\alpha_i= \beta_i\wedge \gamma_i,$ notice that 
$$\iota_Xd_{\mathfrak g}\alpha_i =  \iota_X(\beta_i\wedge \gamma_i) = \beta_i(X)\gamma_i - \gamma_i(X)\beta_i.$$ 
Therefore,
\begin{align*}
& \sum_{i=1}^n (\langle [\alpha_1, \ldots, \iota_Xd_{\mathfrak g}\alpha_i, \ldots,  \alpha_n], Y\rangle -\langle [\alpha_1, \ldots, \iota_Yd_{\mathfrak g}\alpha_i, \ldots,  \alpha_n], X\rangle)\\
& = \sum_{i=1}^n (\beta_i(X)\langle [\alpha_1, \ldots, \gamma_i, \ldots,  \alpha_n], Y\rangle - \gamma_i(X)\langle [\alpha_1, \ldots, \beta_i, \ldots,  \alpha_n], Y\rangle)\\
& - \sum_{i=1}^n (\beta_i(Y)\langle [\alpha_1, \ldots, \gamma_i, \ldots,  \alpha_n], X \rangle - \gamma_i(Y)\langle [\alpha_1, \ldots, \beta_i, \ldots,  \alpha_n], X\rangle)\\
& = \sum_{i=1}^n(\beta_i\wedge [\alpha_1, \ldots, \gamma_i, \ldots,  \alpha_n]+ [\alpha_1, \ldots, \beta_i, \ldots,  \alpha_n]\wedge \gamma_i)(X, Y)\\
& = \sum_{i=1}^n[\alpha_1, \ldots, \beta_i \wedge\gamma_i, \ldots,  \alpha_n](X, Y)\\
& = \sum_{i=1}^n[\alpha_1, \ldots, d_{\mathfrak g}\alpha_i, \ldots,  \alpha_n](X, Y).
\end{align*}
Hence, $d_{\mathfrak g}([\alpha_1, \ldots,  \alpha_n]) = \sum_{i=1}^n [\alpha_1, \ldots, d_{\mathfrak g}\alpha_i, \ldots,  \alpha_n].$ 

Conversely, suppose that $\mathfrak g$ is a Lie algebra for which there is an $n$-ary operation $[~, \ldots, ~]$, on $\mathfrak g^\ast$ making it an $n$-Lie algebra. Furthermore, assume that for any $\alpha_1, \ldots, \alpha_n\in \mathfrak g^*,$
$$d_{\mathfrak g}([\alpha_1, \ldots,  \alpha_n]) = \sum_i [\alpha_1, \ldots, d_{\mathfrak g}\alpha_i, \ldots,  \alpha_n]$$ holds. Then it is easy to check that $\delta$ is a $1$-cocycle. Therefore, the pair $(\mathfrak g, \mathfrak g^\ast)$ is a Lie-Filippov bialgebra of order $n.$ 
This completes the proof.\end{proof} 

\begin{exam}
Observe from Definition \ref{defn-weak-lie-filippov-bialgebroid} that a weak Lie-Filippov bialgebra of order $n$ is a Lie algebra $\mathfrak g$ together with an $n$-ary operation
$$[~,\ldots,~] \colon \underbrace{\mathfrak g^\ast \times \cdots \times  \mathfrak g^\ast}_{n\,\textup{times}} \longrightarrow \mathfrak g^\ast$$ 
which is skew-symmetric, i.e., $[\alpha_{\sigma(1)},\ldots,\alpha_{\sigma(n)}] = \textup{sign}(\sigma) [\alpha_1,\ldots, \alpha_n]$
for $\sigma\in\Sigma_n$ and satisfies the identity 
$$[\alpha_1, \ldots, \alpha_{n-1},[\beta_1, \ldots , \beta_n]] = \sum_{i=1}^n[ \beta_1, \ldots, \beta_{i-1},[\alpha_1,\ldots, \alpha_{n-1}, \beta_i], \ldots, \beta_n],$$ for all  $d_{\mathfrak g}$-closed $\alpha_i \in \mathfrak g^\ast,$ $1 \leq i \leq n-1$ and for any $\beta_j\in \mathfrak g ^\ast,~~ 1\leq j \leq n$ such that
$$d_{\mathfrak g}([\alpha_1, \ldots,  \alpha_n]) = \sum_{i=1}^n [\alpha_1, \ldots, d_{\mathfrak g}\alpha_i, \ldots,  \alpha_n]$$ holds for all $\alpha_1, \ldots , \alpha_n \in \mathfrak g^\ast.$ 

Any Lie-Filippov bialgebra of order $n$ is clearly a weak Lie-Filippov bialgebra of order $n.$ Hence, every Nambu-Lie group of order $n$ gives rise to example of weak Lie-Filippov bialgebroid of order $n.$
\end{exam}

\begin{remark}
Let  $\delta : \mathfrak{g} \rightarrow \bigwedge^2 \mathfrak{g}$ be any $1$-cocycle of the Lie algebra $\mathfrak{g}$ (with representation in $\bigwedge^2 \mathfrak{g}$) of a connected and simply connected Lie group $G.$ It is known \cite{lu-thesis, lu-weinstein } that $\delta$ integrates to a unique multiplicative bivector $\pi_G$ on $G$ such that the intrinsic derivative
$d_e \pi_G$  is $\delta$. In the case when $\delta^* : \bigwedge^2 \mathfrak{g}^* \rightarrow \mathfrak{g}^*$ defines a Lie algebra structure on $\mathfrak{g}^*,$ so that the pair $(\mathfrak{g}, \mathfrak{g}^*)$ is a Lie bialgebra, the multiplicative bivector field turns out to be a Poisson bivector field and hence, defines a Poisson Lie group structure
on $G$.

It is also known \cite{vaisman} that given a connected and simply connected Lie group $G$ with Lie algebra $\mathfrak{g}$, any $1$-cocycle
$\delta : \mathfrak{g} \rightarrow \bigwedge^n \mathfrak{g}$ integrates to a multiplicative $n$-vector field $\Pi$ on $G$. However, even if $\delta^* : \bigwedge^n \mathfrak{g}^* \rightarrow \mathfrak{g}^*$ satisfies the fundamental identity so that $\delta^*$ defines a Filippov algebra structure on $\mathfrak{g}^*,$ the induced $n$-vector field $P$ on $G$ need not be locally decomposable, and hence, need not define a Nambu-Lie group structure on $G.$
\end{remark}

We end with the following remark suggesting a list of some related problems of interest for further investigation.
\begin{remark}\label{problems-interest}
\begin{enumerate}
\item In \cite{liu-xu}, the authors studied the local structure of Lie bialgebroids at regular points. In particular, they classify all transitive Lie bialgebroids. In special cases, they are connected to classical dynamical $r$-matrices and matched pairs induced by Poisson group actions. Following this method, it would be interesting to study local structure of weak Lie-Filippov bialgebroids.

\item Nijenhuis operators have been introduced in the theory of integrable systems in the work of Magri, Gelfand and Dorfman \cite{dorfman}, and, under the name of
hereditary operators, in that of Fuchssteiner and Fokas. Poisson-Nijenhuis structures were
defined by Magri and Morosi in \cite{magri-morosi} in their study of completely integrable systems.
There is a compatibility condition between the Poisson structure and the Nijenhuis structure that is expressed by the vanishing of a rather complicated tensorial expression. 
In \cite{kosmann-sch}, the author expressed this condition in a very simple way, using
the notion of a Lie bialgebroid. More precisely, the author proved that a Poisson structure and a Nijenhuis structure constitute a Poisson-Nijenhuis structure
if and only if the following condition is satisfied: the cotangent and tangent bundles are
a Lie bialgebroid when equipped respectively with the bracket of $1$-forms defined by the
Poisson structure, and with the deformed bracket of vector fields defined by the Nijenhuis
structure. In the context of Nambu-Poisson manifold, it would be interesting to investigate whether the Nambu form-bracket of $1$-forms and the deformed bracket of vector fields defined by the Nijenhuis structure leads to a notion of Nambu-Nijenhuis structure and in case the answer is positive, then to give a physical interpretation of compatibility condition of the associated weak Lie-Filippov bialgebroid.

\item As remarked in the introduction, given a Lie bialgebroid $(A, A^\ast)$ over a smooth manifold $M,$ the direct sum bundle $A\oplus A^\ast$ carries a Courant algebroid structure such that $A$, $A^\ast$ are both Dirac subbundles of $A\oplus A^\ast$ \cite{liu-weinstein-xu}. Conversely, if $L_1$ and $L_2$ are Dirac subbundles of a Courant algebroid $E$ which are transversal to each other, then the pair $(L_1, L_2)$ is a Lie bialgebroid, where $L_2$ is considered as the dual bundle of $L_1$ under the non degenerate pairing of the Courant algebroid $E.$  Thus, the authors extended the theory of Manin triples from Lie bialgebras to Lie bialgebroids. It would be interesting to find out how far this theory can be extended replacing Lie bialgebroid by weak Lie-Filippov bialgebroid.
\end{enumerate}
\end{remark}

\mbox{ }\\

\providecommand{\bysame}{\leavevmode\hbox to3em{\hrulefill}\thinspace}
\providecommand{\MR}{\relax\ifhmode\unskip\space\fi MR }
\providecommand{\MRhref}[2]{%
  \href{http://www.ams.org/mathscinet-getitem?mr=#1}{#2}
}
\providecommand{\href}[2]{#2}

\end{document}